\date{September 6, 2012}

\documentclass[11pt,a4paper]{article}
 
\usepackage[utf8]{inputenc}
 
\usepackage{array,amsbsy,amscd,amsfonts,color,amssymb,amstext,amsmath,latexsym,amsthm,amscd
,multicol,graphicx,makeidx,tikz}
\usepackage[english]{babel}
\usepackage{hyperref}

\numberwithin{equation}{section}
 
\makeindex
 
\newtheorem{theorem}{Theorem}
\newtheorem{proposition}[theorem]{Proposition}

\newtheorem{lemma}[theorem]{Lemma}
\newtheorem{corollary}[theorem]{Corollary}

\theoremstyle{definition}
\theoremstyle{definition}\newtheorem{definition}[theorem]{Definition}
\theoremstyle{definition}\newtheorem{remark}[theorem]{Remark}
\theoremstyle{definition}
\theoremstyle{definition}
\theoremstyle{definition}
\theoremstyle{definition}
\numberwithin{theorem}{section}

\def\proofof [#1] {\noindent {\bf Proof of #1. } }

\def\al #1.{{\mathcal{#1}}}

\newcommand{\KK}{\mathfrak{K}_\A}

\newcommand{\A}{\mathcal{A}}
\newcommand{\Ring}{\mathcal{R_A}}
\newcommand{\RRing}{\mathcal{\tilde{R}_A}}
\newcommand{\B}{\mathcal{B}}

\newcommand{\ZZ}{\mathcal{Z}}
\newcommand{\I}{\mathcal{I}}
\newcommand{\K}{\mathcal{K}}
\renewcommand{\H}{\mathcal{H}}

\newcommand{\N}{\mathbb{N}}
\newcommand{\NN}{\mathbb{N}_0}
\newcommand{\R}{\mathbb{R}}
\newcommand{\C}{\mathbb{C}}

\renewcommand{\S}{S^1}

\newcommand{\Z}{\mathbb{Z}}

\newcommand{\unit}{\mathbf{1}}

\newcommand{\bp}{\begin{proof}}
\newcommand{\ep}{\end{proof}}
\newcommand{\bdp}{\begin{dproof}}
\newcommand{\edp}{\end{dproof}}
\newcommand{\ra}{\rightarrow}

\newcommand{\Diff}{\operatorname{Diff}(\S)}

\newcommand{\PSL}{\operatorname{PSL(2,\R)}}
\newcommand{\Mob}{\operatorname{PSL(2,\R)}}

\newcommand{\PSI}{\operatorname{PSL(2,\R)}^{(\infty)}}
\newcommand{\PST}{\operatorname{PSL(2,\R)}^{(2)}}

\newcommand{\Mat}{\operatorname{M}}

\newcommand{\Ad}{\operatorname{Ad }}

\newcommand{\Alg}{\operatorname{Alg}}

\newcommand{\rmi}{\operatorname{i}}

\newcommand{\irr}{\operatorname{irr}}

\newcommand{\rme}{\operatorname{e}}
\newcommand{\red}{\operatorname{red}}
\newcommand{\locn}{\operatorname{ln}}

\renewcommand{\ker}{\operatorname{ker }}

\newcommand{\End}{\operatorname{End}}

\newcommand{\id}{\operatorname{id}}

\newcommand{\ie}{{i.e.,\/}\ }

\newcommand{\cf}{{cf.\/}\ }

\title{\huge Representations of Conformal Nets, Universal C*-Algebras and K-Theory}
\author{
\phantom{X}\\
{\sc Sebastiano Carpi}$^{1}$\footnote{Supported in part by the ERC
Advanced Grant 227458 "Operator Algebras and Conformal Field Theory"},
{\sc Roberto Conti}$^{2}$,
{\sc Robin Hillier}$^{3*}$,
{\sc Mih\'aly Weiner}$^{4*}$\\
\phantom{X}\\
${}^1$ Dipartimento di Economia,
Universit\`a di Chieti-Pescara ``G. d'Annunzio''\\
Viale Pindaro, 42, I-65127 Pescara, Italy\\
E-mail: {\tt s.carpi@unich.it}\\
\phantom{X}\\
${}^2$ Dipartimento di Scienze di Base e Applicate per l'Ingegneria\\ 
Sezione di Matematica, Sapienza Universit\`a di Roma \\
Via A. Scarpa, 16, I-00161 Roma, Italy\\
E-mail: {\tt roberto.conti@sbai.uniroma1.it}\\
\phantom{X}\\
${}^3$
Dipartimento di Matematica,
Universit\`a di Roma ``Tor Vergata''\\
Via della Ricerca Scientifica, 1, I-00133 Roma, Italy\\
E-mail: {\tt hillier@mat.uniroma2.it}\\
\phantom{X}\\
${}^4$ Mathematical Institute, Department of Analysis, \\
Budapest University of Technology \& Economics (BME) \\
M\"uegyetem rk. 3-9, H-1111 Budapest, Hungary \\
E-mail: {\tt mweiner@renyi.hu}\\
}
 
\begin{document}
\maketitle
 
\begin{abstract} 

We study the representation theory of a conformal net $\A$ on $S^1$ from a K-theoretical point of view using its universal C*-algebra $C^*(\A)$. We prove that if $\A$ satisfies the split property then, for every representation $\pi$ of $\A$ with finite statistical dimension, 
$\pi(C^*(\A))$ is weakly closed and hence a finite direct sum of type I$_\infty$ factors. We define the more manageable locally normal universal 
C*-algebra $C_{\rm ln}^*(\A)$ as the quotient of $C^*(\A)$ by its largest ideal vanishing in all locally normal representations and we investigate its structure. In particular, if $\A$ is completely rational with $n$ sectors, then $C_{\rm ln}^*(\A)$ is a direct sum of $n$ type I$_\infty$ factors. Its ideal $\KK$ of compact operators has nontrivial K-theory, and we prove that the DHR endomorphisms of  
$C^*(\A)$ with finite statistical dimension act on 
$\KK$, giving rise to an action of the fusion semiring of DHR sectors on $K_0(\KK)$. Moreover, we show that this action corresponds to the regular representation of the associated fusion algebra. 

\end{abstract}

\section*{Introduction}

Conformal field theory (CFT) is the theory of quantum fields with conformal 
symmetry, see e.g. \cite{DMS96}. This subject is interesting on its own and because it provides 
further insight into the structural aspects of more general quantum field theories. 
Moreover it is deeply related with many areas of theoretical physics (e.g. string theory, critical phenomena)
and of mathematics (e.g. number theory, topology of three-dimensional manifolds, infinite-dimensional Lie algebras and Lie groups, quantum groups, subfactors). Of particular importance is the two-dimensional case. Here, the chiral fields, namely the fields depending on one light ray coordinate only, play a special role and one is led to study the corresponding chiral theories on the light ray $\R$ or more frequently on its compactification $S^1$ where the conformal symmetry naturally acts.

The operator algebraic approach to two-dimensional chiral CFT (the case we are going to study here) goes through conformal nets on $S^1$, 
namely inclusion-preserving maps 
\[
  I \mapsto \A(I)
\]
from the set of proper (nonempty, nondense, open) intervals of the unit circle $\S$ into the family of von Neumann algebras (actually type III$_1$ factors) acting on a fixed separable Hilbert space and covariant under a given representation of the conformal group, 
\cf \cite{FRS2,FG,GL2}. Although there are various important differences mainly due to the low-dimensional spacetime topology, these are the one-dimensional analogues of the nets of local von Neumann algebras on Minkowski spacetime (Haag-Kastler nets) which are the basic ingredients in the so called algebraic quantum field theory (local quantum physics),  see \cite{Haag} for a standard reference book. As a consequence, besides the powerful functional analytic methods of the theory  of operator algebras, the operator algebraic approach allows the study of CFT through the theory of superselection sectors of Doplicher, Haag and Roberts (DHR), \cite{DHR71,DHR74}.

From the beginning the operator algebraic approach to CFT has shown to provide a natural framework for classification \cite{BMT88} and structural analysis \cite{FRS2, GL1, FG, Reh}. A central role has been played by the deep relations with the theory of subfactors initiated by Jones in \cite{Jo83} and the Tomita-Takesaki modular theory of von Neumann algebras \cite{Tak70}. On the one hand the connections with the theory of subfactors have occurred through the analysis of the representations of the braid group which emerge in the study of the DHR statistics in low spacetime dimensions
\cite{FRS89,FRS2} and, more explicitly, through the relationship between the Jones index and the statistical dimension of superselection sectors discovered by Longo \cite{Lon89,Lon90}. On the other hand, as shown by Brunetti, Guido and Longo \cite{BGL93} and (independently) by Fr\"{o}hlich and Gabbiani, \cite{FG} the action of the modular group of local von Neumann algebras with respect to the vacuum state, can always be described in terms of the conformal symmetry.
Subsequently this approach produced a large number of remarkable results not only for the study of CFT but also for its impact on the theory of subfactors. A very important step in these later developments was made by Wassermann \cite{Was98} with the computation of the Connes fusion for the positive energy representations 
of the loop groups $L{\rm SU(N)}$. As a consequence of this work the composition of DHR sectors of the associated conformal nets turns out to correspond to the Verlinde fusion. On the basis of this result many other DHR fusion rings have been subsequently determined and found out to be isomorphic to the corresponding Verlinde rings. It also served as a starting point for the work of Xu \cite{Xu00} on Jones-Wassermann subfactors for disconnected intervals which provided inspiration and examples for the notion of completely rational nets (conformal net analogues of rational chiral conformal field theories) introduced shortly after by Kawahigashi, Longo and M{\"u}ger \cite{KLM}.
We refrain from giving further details on the various achievements of the operator algebraic approach to CFT but we just mention two of the most representative recent results: the classification of conformal nets with central charge $c <1$ by Kawahigashi and Longo \cite{KL04} and the development of the technique of mirror extensions for the construction of new nets \cite{Xu07}. 

Although various areas in the theory of operator algebras, in particular the Tomita-Takesaki modular theory for von Neumann algebras and the 
theory of subfactors, have played a central role, noncommutative geometry \cite{C-book} and K-theory for operator algebras  \cite{Bl} experienced -- with few exceptions -- only marginal 
considerations in the conformal nets setting. In \cite{Lo01} Longo suggested to consider quantum field theories and their superselection sectors as quantum analogue of infinite-dimensional manifolds and elliptic operators respectively in order to look for a geometrical interpretation of  the Jones index for subfactors through the statistical dimension of superselection sectors (quantum index theorem). 
In \cite{KL05} Kawahigashi and Longo, with the above picture as a background, studied the asymptotic behavior of the characters of completely rational conformal nets on $S^1$ guided by the classical analogy of Weyl's asymptotics for the trace of the heat kernel. Later, Carpi, Kawahigashi and Longo \cite{CKL08} have began a systematic study of local nets on $S^1$ with superconformal symmetry (superconformal nets). One of the motivations for their work concerned the relations with the noncommutative geometrical framework of Connes. In particular they proved a formula  relating the Fredholm index of the (upper off diagonal part of) the supercharge operator in Ramond representations to the Jones index of subfactors and they proposed to study the spectral triples having these supercharge operators as Dirac operators together with the corresponding JLO cocycles in entire cyclic cohomology (Chern characters in quantum K-theory) \cite{JLO1}. An important step in this direction has been done by 
Carpi, Hillier, Kawahigashi and Longo in \cite{CHKL10} with the definition of nets of spectral triples associated to the unitary representations of the (N=1) super-Virasoro algebras. This requires the solution of nontrivial technical domain problems. In the introduction of the same paper
(see also the outlook) an outline of a general ``noncommutative geometrization" program for CFT is also given. This program has been recently carried out by Carpi, Hillier and Longo \cite{CHL12} through a general procedure which allows to associate spectral triples and hence entire cyclic cohomology classes of suitably chosen global algebras of differentiable operators associated to the given net and its superselection sectors. 
Although this procedure gives rise to serious domain problems it is shown by using the index pairing with K-theory \cite{C-book} that the cohomology classes associated to superselection sectors actually provide nontrivial geometric invariants. 

At this point a natural question arises. Would it be possible to take a dual (topological) point of view and study the superselection sectors directly in K-theory without using the spectral triples (differentiable structure) and the corresponding entire cyclic cohomology classes as a starting point? Such an approach would seem to involve various advantages: it would eliminate the need to restrict to nets on $S^1$ with superconformal symmetry; it would avoid the serious domain problems inherent in the use  of Connes' spectral triples; it could give further insights into the spectral triples approach; it could shed light on possible connections with other situations where K-theory appears in quantum field theory, see e.g. \cite{Free02}. In this paper we propose a first investigation in this direction. 

In order to explain some of the main ideas underlying this work  we briefly describe the analogous situation on four-dimensional 
Minkowski spacetime, for which the DHR theory of superselection sectors was originally formulated.  Here one deals with nets 
${\mathcal O} \mapsto {\mathfrak A}({\mathcal O})$ of von Neumann algebras where ${\mathcal O}$ runs over the set of spacetime double cones. In order to capture the global nature of the superselection sectors one has to define a
suitable global  C*-algebra. To this end, since the  set of double cones is directed under inclusion, one can consider the  C*-inductive limit of the local von Neumann algebras 
${\mathfrak A}({\mathcal O})$, which here, as usual, will be denoted by ${\mathfrak A}$ as the underlying net. The superselection sectors of the theory are realized as DHR 
(i.e. localized and transportable) unital endomorphisms of the  C*-algebra ${\mathfrak A}$. As a consequence the DHR-sectors act on the K-theory  of ${\mathfrak A}$ and one might be tempted to use this action in order to obtain K-theoretical invariants for the DHR sectors. 
To simplify our discussion we assume the split property which is expected to hold for any physically reasonable net (see  \cite[V.5]{Haag}). 
Then it turns out that the quasi-local  C*-algebra ${\mathfrak A}$ is a simple  C*-algebra with trivial K-theory and it is moreover independent of the underlying net (cf. Proposition \ref{PropQuasilocal}). In particular the action of DHR sectors on the K-theory of ${\mathfrak A}$ gives no information at all. 
 
For a conformal net $\A$ on $\S$ the situation is different: although the set of intervals is not directed, one can define in a suitable way a universal C*-algebra $C^*(\A)$. It is generated by the local algebras associated to all proper intervals of $\S$ and the equivalence classes of its irreducible locally normal representations are in one-to-one correspondence with the superselection sectors of the net.
It therefore contains a 
lot of information about the net $\A$, \cf \cite{Fre90,FRS2}. As in the case of the quasi-local 
 C*-algebra the superselection sectors can be realized as DHR endomorphisms of $C^*(\A)$, but in general $C^*(\A)$ is no longer simple and the vacuum 
representation is not faithful either. Moreover, as an abstract C*-algebra, it depends on the representation theory of the underlying net (cf. Proposition \ref{prop:U-propsCA}).

 Unfortunately we are not able to determine $C^*(\A)$  or its K-theory even in special cases. 
The problem is that the definition of $C^*(\A)$ involves plenty of representations of $\A$ which are not locally normal and thus difficult to handle. 
On the other hand these non-locally normal representations do not appear to have direct relevance for CFT. It seems therefore natural to 
consider a related global  C*-algebra here called the locally normal universal  C*-algebra of the net $\A$ and denoted by 
$C^*_{\rm ln}(\A)$. It can be defined as the quotient of $C^*(\A)$ by the largest of its ideals vanishing in all locally normal representations and 
it admits by definition faithful locally normal representations. Moreover, if one restricts to locally normal representations of the net $\A$ then 
 $C^*_{\rm ln}(\A)$ has all the relevant properties of $C^*(\A)$. In particular the superselection sectors of $\A$ can also be realized as  DHR-endomorphisms of $C^*_{\rm ln}(\A)$.  
 
One might be tempted to believe that the nontrivial topology of $S^1$ may result in a nontrivial K-theory for the algebra $C^*_{\rm ln}(\A)$.
However, one of the main results of this paper is the proof of the following surprising fact: for every completely rational conformal net $\A$ on $S^1$ the locally normal universal  C*-algebra $C^*_{\rm ln}(\A)$ is isomorphic to a finite direct sum of countably decomposable infinite-dimensional type I factors (one for each sector of $\A$), see Theorem 
\ref{piC*A_rational}. 
As a consequence  $K_0(C^*_{\locn}(\A))=0=K_1(C^*_{\locn}(\A))$ (Theorem \ref{K=0}). 
  Actually the above result follows from a more general fact (Theorem \ref{piC*A_reducible}): if $\A$ is a conformal net on $S^1$ satisfying the split property (but not necessarily completely rational) then $\pi(C^*(\A))$ is weakly closed for every locally normal representation $\pi$ with finite statistical dimension. 
 
 In order to overcome the problem of trivial K-theory we consider, for a given completely rational net $\A$, the largest norm-separable ideal 
 $\KK$ of  $C^*_{\rm ln}(\A)$.  If $\A$ has $n$ sectors then 
 $\KK$ is isomorphic to the direct sum of $n$ copies of the  C*-algebra $\K(\H)$ of compact operators on a separable infinite-dimensional Hilbert space  $\H$. 
We show that the DHR endomorphisms of $C^*_{\rm ln}(\A)$ with finite statistical dimension restrict to endomorphisms of $\KK$ and they give rise to an action of the fusion semiring ${\mathcal R}_\A$ of sectors with finite statistical dimension on $K_0(\KK)$ which corresponds to the regular representation of the fusion algebra generated by ${\mathcal R}_\A$ (Theorem \ref{th:K-sectors}). 

We end this introduction with some comments on possible connections with the work of Freed, Hopkins and Teleman on loop groups and twisted K-theory \cite{FHT11}, cf. also \cite{EG09} for relations with subfactors and modular invariants.  If $G$ is a simply connected compact Lie group and $LG = C^\infty(S^1,G)$ is the corresponding loop group then the projective unitary representation of $LG$ with lowest energy $0$ (vacuum representation) and level $k$ gives rise to a conformal net 
$\A_{G_k}$ on $S^1$, see \cite{FG}.  Assume  that $\A_{G_k}$ is completely rational and that the ring whose elements are the formal differences of elements in the semiring ${\mathcal R}_{\A_{G_k}}$ generated by the DHR sectors of $\A_{G_k}$ is isomorphic to the Verlinde fusion ring $R^k(LG)$ (this is known to be true e.g. for $G={\rm SU}(N)$ and all positive integers $k$ as a consequence of the results in \cite{Was98,Xu00} mentioned above, cf. also the comments at the end of Sect. \ref{sec:prelim} and 
Remark \ref{rem:CFT-univC} (4)).  
As shown in \cite{FHT11} $R^k(LG)$ is isomorphic to a twisted equivariant K-theory of the Lie group $G$.  Hence our results indirectly exhibit a relation between the above twisted equivariant K-theory of $G$ and the K-theory of the noncommutative  C*-algebra $\KK$ for $\A=\A_{G_k}$. 

\section{Preliminaries on conformal nets}\label{sec:prelim}
 
We fix the following notation: $\N$ and $\NN$ stand for the positive integers and the nonnegative integers, respectively. Given a (complex) Hilbert space $\H$, $B(\H)$ will stand for the 
C*-algebra of bounded linear operators on $\H$, $\K(\H)$ (or sometimes simply $\K$, if $\H$ is infinite-dimensional separable) for the 
C*-subalgebra of compacts on $\H$, $\otimes$ the minimal (spatial) tensor product of 
C*-algebras (\ie the completion with respect to the minimal  C*-norm on 
the algebraic tensor product), and $\bar{\otimes}$ the (weakly closed) 
tensor product of von Neumann algebras. We say that two $*$-isomorphisms $\phi_i : A \to B$, $i=1,2$ of the  C*-algebra $A$ into the unital  C*-algebra $B$ are unitarily equivalent in $B$ if there is a unitary $u \in B$ such that $\phi_1 = \Ad(u) \circ \phi_2$, where 
$\Ad(u)(b)  := ubu^*$ for all $b \in B$.

When composing homomorphisms of algebras or groups, we will usually drop the sign `` $\circ$ " between the two homomorphisms. Finally, we denote by ${\mathcal Z} (A)$ the center of any associative algebra $A$.

\subsection*{Möbius covariant nets}
 
Let $\I$ be the set of open, nonempty and nondense intervals of the unit circle 
$S^1 =\{z\in \C : \,|z|=1 \}$. Diffeomorphisms
of $S^1$ of the form $z \mapsto \frac{az+b}{\overline{b}z+\overline{a}}$ 
with $a,b\in \C$, $|a|^2-|b|^2=1$, are called {\it 
Möbius transformations}; they form a Lie group isomorphic to 
$\PSL$, and we shall identify these two groups henceforth. In particular, rotations are Möbius transformations and we shall use the symbol $R_\alpha$ to denote the anticlockwise rotation of $S^1$ 
by an angle $\alpha$.
 
A \emph{Möbius covariant net} $\A$ over $\S$ (\cf \cite{FG,GL2})
consists of a family of von Neumann algebras $\A(I)$ ($I\in\I$)
acting on a common separable Hilbert space $\H$ together with a given strongly 
continuous representation $U$ of $\PSL$ satisfying 
\begin{itemize}
\item {\it isotony}: $\A(I_1)\subset \A(I_2)$ if $I_1\subset I_2$,
\item {\it locality}: elements of $\A(I_1)$ commute with those of 
$\A(I_2)$ whenever $I_1\cap I_2 = \emptyset$,
\item {\it covariance}: $U(g)\A(I)U(g)^*=\A(g(I))$ for all $g\in \PSL$
and $I\in \I$,
\item
{\it positivity of the energy}:
the conformal Hamiltonian $L_0$, defined by 
the equation $U(R_\alpha)=e^{i\alpha L_0}$ ($\forall \alpha \in \R$), 
is positive,
\item
{\it existence, uniqueness and cyclicity of the vacuum:}
up to phase there exists a unique unit vector
$\Omega \in \H$ called the ``vacuum vector''
which is invariant under the action of $U$; moreover, 
it is cyclic for the von Neumann
algebra $\bigvee_{I\in\I}\A(I)$.
\end{itemize}
There are many known important consequences of the above definition. In 
particular, we shall often use the following ones (for proofs see \cite{FG,FJ96,GL2}).
{\it Reeh-Schlieder property:}
$\Omega$ is a cyclic and separating vector of the algebra $\A(I)$,
for every $I \in \I$. 
{\it Haag-duality:}
$\A(I)'= \A(I')$ for every $I \in \I$,
where $I'\in\I$ denotes the interior of the complement set of
$I$ in $S^1$. 
{\it Irreducibility}:
$\bigvee_{I\in\I}\A(I) = {\rm B}(\H)$.
{\it Factoriality}: (excluding the case dim$(\H)=1$ 
in which the net is trivial)
$\A(I)$ is a type III$_1$ factor
for every $I \in \I$. 
{\it Additivity}: if $I \subset \bigcup_{\beta} I_\beta$ 
then $\A(I)\subset \bigvee_{\beta}\A(I_\beta)$.
 
There are some further properties of $\A$ which do {\it not} follow from 
the above axioms, 
but are often satisfied in the physically interesting cases. 
{\it Strong additivity}: $\A(I)=\A(I_1)\vee \A(I_2)$ whenever 
$I_1,I_2,I\in \I$ are such that the closure of $I$ coincides with the 
closure of $I_1\cup I_2$. 
{\it Split property}: for any inclusion $\bar{I}_-\subset I_+$, there is a 
type I factor $F$ such that $\A(I_-)\subset F \subset \A(I_+)$; there is actually a canonical choice of $F$, \cf \cite{DL}, but this is irrelevant in the present publication.
{\it Complete rationality}: the net is strongly 
additive, split and its $\mu$-index, that is the index of the so-called ``2-interval inclusions'',
is finite. The $\mu$-index of a M\"{o}bius covariant net $\A$ on $S^1$ with the split property is defined as follows: let 
$I_1, I_2, I_3,I_4 \in \I$ be four intervals in anticlockwise order, obtained by removing four points from $S^1$. Then  
$\A(I_1) \vee \A(I_3) \subset (\A(I_2) \vee \A(I_4))'$ is an irreducible inclusion of type III factors (``2-interval inclusion'')
whose index $\left[(\A(I_2) \vee \A(I_4))' :\A(I_1) \vee \A(I_3)\right]$ does not depend on the choice of the intervals and is called the $\mu$-index of $\A$, cf. \cite{KLM} for more details.

\subsection*{Conformal nets}
 
Let $\Diff$ be the group of orientation preserving (smooth) 
diffeomorphisms of $S^1$, which we shall consider with the
usual $C^\infty$-topology. In what follows, we shall consider strongly 
continuous projective unitary representations of $\Diff$.
If $V$ is such a representation, then sometimes for a $g \in \Diff$ 
we shall think of $V(g)$ as a
unitary operator. Although there is more than one way to ``fix phases'',
note that expressions like $\Ad (V(g))$ or $V(g) \in M$ for a von
Neumann algebra $M \subset {\rm B}(\H)$ are unambiguous. 
We shall also say that $V$ is an extension of the unitary representation 
$U$ of $\Mob$ if we can arrange the phases in such a way that
$V(g)=U(g)$, or without mentioning phases: $\Ad(V(g))=\Ad(U(g))$, for all $g \in \Mob$.
 
A \emph{conformal net} is a Möbius covariant net $(\A,U)$ 
such that $U$ extends to a 
continuous projective unitary representation of $\Diff$ denoted again by $U$ and satisfying
\begin{itemize}
\item
$U(g)\A(I)U(g)^* = \A(gI),$
\item
$g|_I={\rm id}_I \Rightarrow
\Ad (U(g))|_{\A(I)}=\rm{id}_{\A(I)}$,
\end{itemize}
for all $g \in \Diff$ and $I \in \I$. Note that though there are some 
``pathological '' cases in which such an extension does not exist, when it 
{\it does} exist, it is unique, \cf \cite{CW05,Wei05}.
Note also that by Haag-duality, if a diffeomorphism is
localized in the interval $I$ --- i.e. it acts trivially elsewhere --- then, by the second listed property, the corresponding unitary is also localized in $I$ in the sense that it belongs to $\A(I)$. From now on all nets are supposed to be conformal.

\subsection*{Representations of conformal nets}
 
A \emph{representation} of $\A$ is a 
family $\pi=(\pi_I)_{I\in\I}$ of (unital) $*$-representations $\pi_I$ of $\A(I)$ on a 
common Hilbert space $\H_\pi$ such that $\pi_{I_2|\A(I_1)}=\pi_{I_1}$ 
whenever $I_1\subset I_2$. The representation $\pi$ is called 
{\it locally normal} if $\pi_I$ is normal for every $I\in \I$; 
by \cite[Theorem 5.1]{Tak79}, this is always the case if $\H_\pi$ is separable. Conversely if $\pi$ is a cyclic representation of $\A$, i.e.
$\bigvee_{I \in\I} \pi_I(\A(I))$ has a cyclic vector in $\H_\pi$, then $\H_\pi$ is separable as a consequence of the fact that  $\A(I)$ has separable predual for all $I\in \I$. Accordingly a representation $\pi$ of $\A$ is locally normal if and only if it is a direct sum of separable representations.  
The \emph{vacuum representation}  $\pi_0$ on the vacuum Hilbert space $\H_{\pi_0}  := \H$ is defined by 
$\pi_{0 , I}(x) = x$, for all $I\in \I$ and all $x\in \A(I)$ and it is obviously locally normal.    

Two representations $\pi_1$ and $\pi_2$ of the same net
are {\it equivalent} if there exists a unitary operator $u: \H_{\pi_1} \to \H_{\pi_2}$ such that 
$u\pi_{1,I}(x)=\pi_{2,I}(x)u$ for every $I\in\I$ and $x\in \A(I)$. We denote by $[\pi]$ the unitary equivalence class of a
representation $\pi$ of $\A$. 
The representation $\pi$ is called {\it irreducible}, if $\cap_{I\in\I} \pi_I(\A(I))' = \C \unit_{\H_\pi}$. 
The equivalence classes of irreducible locally normal representations of $\A$ are  called \emph{sector} (or superselection sector or DHR sector). 
The vacuum representation $\pi_0$ is irreducible and the corresponding sector $[\pi_0]$ is called the \emph{vacuum sector}. 

Let $\pi$ be a representation of $\A$ and let $I_1, I_2 \in \I$ be such that $I_2 \subset I_1'$. If $I_2 \neq I_1'$ then there is an interval $I_0\in\I$ containing $I_1 \cup I_2$. Hence $\pi_{I_1}(\A(I_1)) =  \pi_{I_0}(\A(I_1))$ commutes with $\pi_{I_2}(\A(I_2)) =  \pi_{I_0}(\A(I_2))$. 
Moreover, for every $I\in \I$, $\pi_I$ is faithful because $\A(I)$ is a countably decomposable type III factor and hence a simple  C*-algebra. 

If $\pi$ is locally normal then $\pi_I(\A(I))$ is a type ${\rm III}_1$ factor with separable predual for all $I\in \I$.  Moreover it  
follows easily from the definition and additivity of conformal nets that  if $I \subset \bigcup_{\beta} I_\beta$, $I, I_\beta \in \I$, then 
$\pi_I(\A(I)) \subset \bigvee_{\beta} \pi_{I_\beta}(\A(I_{\beta}))$. As a consequence we also have $\pi_I(\A(I)) \subset \pi_{I'}(\A(I'))'$ 
for every $I \in \I$ a property which in principle may fail for representations which are not locally normal. 

A representation $\pi$ on the vacuum Hilbert space $\H$ is said to be localized  in $I_0$ if $\pi_{I'} = \id$.  
In this case, by Haag duality, we have $\pi_I(\A(I)) \subset \A(I)$, for all $I\in \I$ containing $I_0$,   i.e. $\pi_I$ is an endomorphism of $\A(I)$. 
A representation of $\A$ localized in some interval in $\I$ is also called a \emph{localized endomorphism} or 
\emph{DHR endomorphism} of the net $\A$. 

If $\pi$ is any representation of $\A$ on a separable Hilbert space $\H_\pi$ and $I_0$ is any interval in $\I$ then there exists a representation localized in $I_0$ and unitarily equivalent to 
$\pi$.

A locally normal representation $\pi$ is always covariant with respect to the Möbius group in the sense that there exists a unique 
strongly continuous unitary representation $U_\pi$ of the universal covering $\PSI$ such that
\begin{equation}
U_\pi(g)\pi_I(a)U_\pi(g)^* = \pi_{\dot{g}I}(U(\dot{g})aU(\dot{g})^*), \quad g \in \PSI \; a \in \A(I),
\end{equation}
where $\dot{g}$ is the image of $g$ in $\PSL$ under the covering map,  and such that 
$U_\pi(g)\in \bigvee_{I\in\I}\pi_I(\A(I))$, for all $g\in\PSI$, \cf \cite{DFK}. This way, the infinitesimal generator $L_0^\pi$ of the lifting of the rotation subgroup turns out to be well-defined and positive \cite{Wei06}.

As a consequence of covariance if $\H_\pi$ is separable, the minimal index \linebreak
$[ \pi_{I'}(\A(I'))': \pi_I(\A(I))]$ is independent of $I$. Its square root 
is called the statistical dimension of $\pi$, is denoted by $d(\pi)$ and depends only on the unitary equivalence class $[\pi]$ of $\pi$. 
Accordingly $d(\pi)=1$ iff  $\pi_{I}(\A(I)) = \pi_{I'}(\A(I'))'$ for all $I \in \I$ i.e. if the representation satisfies Haag duality. In particular $d(\pi_0)=1$. 
If $\H_{\pi_1}$ and $\H_{\pi_2}$ are separable then $d(\pi_1 \oplus \pi_2) =d(\pi_1) +d(\pi_2)$. If the representation $\pi$ is localized in $I_0\in \I$ then, if $I\in \I$ contains $I_0$ we have  $d(\pi)= [\A(I):\pi_I(\A(I))]^{\frac12}$, i.e. $d(\pi)^2$ is the index of the unital endomorphism 
$\pi_I \in \End (\A(I))$.

The sum of $d(\pi)^2$ over all sectors $[\pi]$ is called the global index of $\A$. The (finite) $\mu$-index of a completely rational net coincides with the global index \cite{KLM} and hence such a net has only finitely many sectors, each with finite statistical dimension. 
Finally, for conformal nets, the split property implies equality of the $\mu$-index and the global index, 
and split together with finiteness of the global index imply complete rationality \cite{LX}.

 Many interesting examples of unitary rational chiral CFT models have been proved to correspond to  
completely rational conformal nets including various loop groups models (e.g. $L\operatorname{SU}(n)$ for every positive integer level 
$k$) and the corresponding cosets models, CFTs associated to the discrete series representations of the Virasoro and super-Virasoro algebras, the moonshine vertex operator algebra $V^\natural$ and lattice models. The class of completely rational conformal nets is moreover closed under tensor product, irreducible extensions, orbifold construction and finite index subnets, see e.g. 
\cite{CKL08,DX06,KL04,KL06,Lon03,Xu00b, Xu05}. 

\section{The universal C*-algebra and its image in locally normal representations}
\label{sec:universal}
 
As mentioned above, the single algebras $\A(I)$, $I\in\I$, are all isomorphic, so we cannot 
expect them to carry any information about the net and its 
representations. In fact, the characteristic properties and the 
representation structure of a specific local conformal net lie in the 
family of inclusions between these local algebras. A successful way of 
capturing this structure is to study a global C*-algebra (see \cite{Fre90,FRS2, GL1} and also \cite{RV} for a generalization of this notion)
with certain universality properties.

\begin{definition}\label{def:CFT-univC}
The \emph{universal C*-algebra} of $\A$ is the unique (up to isomorphism) unital C*-algebra $C^*(\A)$ such that
\begin{itemize}
\item[-] for every $I\in\I$, there are unital embeddings $\iota_I:\A(I)\ra C^*(\A)$, such that $\iota_{I_2|\A(I_1)} = \iota_{I_1}$ if $I_1\subset I_2 $, and all $\iota_I(\A(I))\subset C^*(\A)$ together generate $C^*(\A)$ as C*-algebra;
\item[-] for every representation $\pi$ of $\A$ on $\H_\pi$, there is a unique representation $\hat{\pi}:C^*(\A)\ra B(\H_\pi)$ such that
\[
\pi_I = \hat{\pi}\circ \iota_I,\quad I\in \I.
\]
\end{itemize}
\end{definition}
 
\begin{remark}\label{rem:CFT-univC}
\begin{itemize}
\item[(1)] To show that $C^*(\A)$ in \cite[ Sect. 5]{FRS2} is well-defined, we may construct it as follows: consider the $*$-algebra $*\A$ generated freely by all $\A(I)$, let $\iota_I$ be the natural inclusions, and consider the C*-seminorm
\[
  ||a|| := \sup_{\pi} ||\pi(a)||_{\H_\pi},\quad a\in *\A,
\]
where the supremum is taken over all $^*$-representations $\pi$ of $*\A$ such that \linebreak
$\pi\circ\iota_{I_2|\A(I_1)} = \pi\circ\iota_{I_1}$ if $I_1\subset I_2$ (it is easy to show that the $\sup$ in the above formula  is finite).
Then the algebra $C^*(\A)$ can be taken as the completion of $*\A / \ker(||\cdot||)$. This is the general way of defining a universal C*-algebra given a suitable family of generators and relations \cite[II.8.3.1]{Bla06}. Note that $C^*(\A)$ can be defined in this way for every M\"{o}bius covariant net, yet in the present paper we shall only consider conformal nets.
\item[(2)] For convenience, when no confusion can arise, we shall often identify $\A(I)$ with its image $\iota_I(\A(I))$ in $C^*(\A)$. Moreover we drop the symbol $\hat{\cdot}$ over $\pi$.
\item[(3)] We call a representation $\pi$ of $C^*(\A)$ \emph{locally normal} if $\pi\circ\iota_I$ is normal, for all $I\in\I$, in other words, if it comes from a locally normal representation $(\pi_I)_{I\in\I}$ of the net $\A$. We say that $\pi$ is localized in $I_0$ if $\H_\pi$ is the vacuum Hilbert space 
$\H$ and $\pi \circ \iota_{I_0} = \id $, in other words if the family $(\pi\circ \iota_I )_{I\in \I}$ is a DHR endomorphism of $\A$ localized in $I_0$. 
There is a natural one-to-one correspondence between representations $\pi$ of $C^*(\A)$ localized in a given interval $I\in\I$ and covariant endomorphisms $\rho$ of $C^*(\A)$ localized in $I$, i.e. such that $\rho \circ \iota_{I'} = {\rm id}_{\A(I')}$. Namely, given $\pi$, $\rho$ is the unique endomorphism of $C^*(\A)$ localized in $I$ satisfying $\pi=\pi_0\circ\rho$ and the covariance condition 
\[
\Ad (z(\pi,g))\circ \rho = \alpha_{\dot{g}}\circ\rho\circ \alpha_{\dot{g}}^{-1}, \quad g\in\PSI,
\]
where $z(\pi,g)$ is the canonical cocycle associated to $\pi$ and $\alpha: \PSL {\rm} \mapsto \rm{Aut}(C^*(\A))$ is the automorphism group corresponding to the M\"{o}bius covariance of the net $\A$, see \cite[ Sect. 8]{GL1} (cf. also \cite{CHL12}). 
Clearly, the localized covariant endomorphism corresponding to the vacuum representation $\pi_0$ of $\A$ is the identity automorphism 
$\id : C^*(\A) \to C^*(\A)$. 
If $\rho_1$ and $\rho_2$ are localized covariant endomorphisms of $C^*(\A)$ then $\pi_0\circ \rho_1$ is equivalent to $\pi_0 \circ \rho_2$ 
iff $\rho_1 = \Ad(u)\circ \rho_2$ for some unitary $u\in C^*(\A)$ (i.e. $\rho_1$ and $\rho_2$ are unitarily equivalent  in $C^*(\A)$).

The statistical dimension $d(\rho)$ of the localized endomorphism $\rho$ of  $C^*(\A)$  is defined to be the statistical dimension $d(\pi_0\circ \rho)$ of the localized representation $\pi_0 \circ \rho$. Note that if $\pi$ is any locally normal representation of $C^*(\A)$ and $\rho$ is a covariant localized endomorphism of  $C^*(\A)$ (not necessarily corresponding to $\pi$) then the representation $\pi \circ \rho$ is locally normal and 
the equivalence class $[\pi \circ \rho]$ depends on $[\pi]$ and  $[\rho]$ only.
\item[(4)] If $\rho_1$ and $\rho_2$ are covariant endomorphisms of $C^*(\A)$ with finite statistical dimension and localized in the same interval 
$I_0 \in \I$ then the composition $\rho_1 \rho_2$ is a covariant endomorphism localized in $I_0$ with statistical dimension 
$d(\rho_1)d(\rho_2) < + \infty$. The equivalence class $[\pi_0 \circ \rho_1 \rho_2]$ depends only on
 $[\pi_0 \circ \rho_1 ]$ and $[\pi_0 \circ \rho_2]$.  As a consequence, with the operations  
 $$[\pi_0 \circ \rho_1 ] [\pi_0 \circ \rho_2]  := [\pi_0 \circ \rho_1 \rho_2] ,\quad 
 [\pi_0 \circ \rho_1] + [\pi_0 \circ \rho_2]  := [\pi_0 \circ \rho_1 \oplus \pi_0 \circ \rho_2]$$ 
 the set of equivalence classes of locally normal representations of $\A$ with finite statistical dimension becomes a unital  semiring 
 (without $0$)  $\Ring$, which is commutative because of the DHR braiding \cite{FRS2,GL2,Reh} and is equipped with a dimension function
  $d([\pi_0 \circ \rho])  := d(\rho)$ and a conjugation $\overline{[\pi_0 \circ \rho]}  := [\pi_0 \circ \bar{\rho}]$.  Here 
  $\bar{\rho}$ is the localized covariant endomorphism conjugate to $\rho$ (see  \cite[Subsect. 2.3.]{GL2} and \cite{LR97}). 
  It is defined up to unitary equivalence in 
 $C^*(\A)$ and satisfies $[\pi_0\circ \bar{\bar{\rho}}]=[\pi_0 \circ \rho]$, namely it is involutive on the equivalence classes. If $\pi_0 \circ \rho$ is irreducible then, up to unitary equivalence in $C^*(\A)$, $\bar{\rho}$ is determined by the irreducibility  of $\pi_0 \circ \bar{\rho}$ and the requirement that $\pi_0$ is equivalent to a subrepresentation of $\pi_0 \circ \bar{\rho}\rho$. If this is the case, the multiplicity of $\pi_0$ as a subrepresentation of 
 $\pi_0 \circ \bar{\rho}\rho$ is $1$.   If $\pi_0 \circ \rho$ is reducible and equivalent to $\pi_0 \circ \rho_1 \oplus \pi_0 \circ \rho_2$ 
 then $\pi_0 \circ \bar{\rho}$ is equivalent to $\pi_0 \circ \bar{\rho_1} \oplus \pi_0 \circ \bar{\rho_2}$. $d(\rho)=d(\bar{\rho})$ for any localized covariant endomorphism  $\rho$ of $C^*(\A)$ with finite statistical dimension.   We call $\Ring$ the \emph{fusion semiring} of $\A$. 
  
Let $\A$ be conformal on $S^1$ with finitely many sectors with finite statistical dimension, $s_1,s_2,\dots s_n$. This is the case if e.g. 
$\A$ is completely rational. Suppose that $s_1=[\pi_0]$ and define $\bar{\cdot}: \{1,2,\dots n \} \to \{1,2, \dots n \}$ by 
$s_{\bar{j}}  := \bar{s}_j$.  
Addition and multiplication in $\Ring$ are related by the so-called \emph{fusion rules}, namely there are coefficients $N_{i j}^k\in\NN$ such that
\begin{equation}\label{eq:prelim-fusion1}
s_i s_j = \sum_{k=1}^n N_{i j}^k s_k.
\end{equation}
From the properties of the conjugation of endomorphisms and commutativity we have 
\begin{equation}\label{eq:prelim-fusion2}
N_{i j}^1=\delta_{i, \bar{j}} ,\quad  N_{i j}^k=N_{j i}^k ,\quad i,j,k =1,\dots,n .
\end{equation}
 As a consequence we also have 
 \begin{equation}\label{eq:prelim-fusion3}
    N_{i j}^{\bar{k}} = N_{j k}^{\bar{i}} = N_{\bar{k} \bar{j}}^{i} ,\quad i,j,k =1,\dots,n
 \end{equation}
 see e.g. \cite[Lemma 2.2.]{Yam99}, cf. also \cite[Sect. 5]{Reh}. 
 
 Fix an interval $I_0 \in \I$. Then the map $[\rho] \mapsto [\pi_0 \circ \rho]$ is a well-defined one to one correspondence 
 between the set 
 of unitary equivalence classes of covariant endomorphisms of $C^*(\A)$ with finite statistical dimension and localized in $I_0$ and the set of equivalence classes of locally normal representations of $C^*(\A)$ with finite statistical dimension. 
Accordingly we may as well view $\Ring$ as the semiring of unitary equivalence classes of covariant endomorphisms of $C^*(\A)$ with finite statistical dimension and localized in any fixed interval $I_0\in \I$.  Then we have $[\rho_1] [\rho_2] = [\rho_1 \rho_2]$, $\overline{[\rho]} = [\bar{\rho}]$, 
$d([\rho]) = d(\rho)$ and $[\rho_1] + [\rho_2] = [\rho]$ iff  $\pi_0 \circ \rho_1 \oplus \pi_0 \circ \rho_2$ is equivalent to 
$\pi_0 \circ \rho$. 

In all computed cases the DHR fusion rules of a completely rational conformal net coincide with the Verlinde fusion rules of the 
the corresponding CFT model, see e.g. \cite[Subsect. 3.2.]{KL05}.

\item[(5)] In the following we shall consider the subalgebra $\Alg(\A)$ of $C^*(\A)$ generated by the local subalgebras 
$\A(I)$, $I\in \I$. Clearly, $\Alg(\A)$ is a unital norm dense $*$-subalgebra of  $C^*(\A)$. 

\item[(6)] Let $I_1, I_2 \in \I$ be such that $I_1 \subset I_2'$. If $I_1 \neq I_2'$ then there is an interval $I_0\in\I$ containing $I_1 \cup I_2$. Hence 
$\iota_{I_1}(\A(I_1)) =  \iota_{I_0}(\A(I_1))$ commutes with $\iota_{I_2}(\A(I_2)) =  \iota_{I_0}(\A(I_2))$ so that we have a weak 
version of locality inside $C^*(\A)$. However, it seems to be unknown whether or not $\iota_{I}(\A(I))$ commutes with $\iota_{I'}(\A(I'))$ 
for $I \in \I$. Accordingly, the strong version of locality may fail in $C^*(\A)$.  On the other hand, if $\pi$ is any locally normal representation of $\A$ then $\pi (\iota_{I}(\A(I))) \subset \pi (\iota_{I'}(\A(I')))'$ for all $I \in \I$ by additivity, cf. Section \ref{sec:prelim}. 

\end{itemize}
\end{remark}

As an abstract C*-algebra  $C^*(\A)$ depends in general on the conformal net $\A$ on $S^1$.  For example if $\A_1$ (resp. $\A_2$) 
is a completely rational conformal net with $n_1$ (resp. $n_2$) sectors then $C^*(\A_1)$ (resp. $C^*(\A_2)$) is a  
C*-algebra with 
$n_1$ (resp. $n_2$) equivalence classes of irreducible representations on a separable Hilbert space. Accordingly, if $n_1 \neq n_2$, 
 $C^*(\A_1)$ and $C^*(\A_2)$ cannot be isomorphic. 
 
 However, as shown by the following proposition the universal  C*-algebra  $C^*(\A)$ has various properties which do not depend on the choice of the conformal net $\A$ on $S^1$.  
      
\begin{proposition}\label{prop:U-propsCA}
Let $\A$ be a conformal net on $S^1$. Then the universal C*-algebra $C^*(\A)$ is unital, properly infinite, has stable rank $\infty$ and it is generated by projections. It is neither separable, nor simple, nor exact, nor purely infinite.
\end{proposition}
\begin{proof} 
Concerning all the relevant definitions and facts about these C*-algebraic properties we have to refer to \cite{Bla06} in order to keep our discussion tight. Moreover, we will use Corollary \ref{corCompacts} which is stated and proved (of course, independently of this proposition)
at the end of this section. 

$C^*(\A)$ is unital by definition. It is nonseparable and properly infinite because it contains the closed unital subalgebras
$\iota_I(\A(I))$, $I \in \I$ which are nonseparable and properly infinite  C*-algebras because they are isomorphic to the type III factors 
$\A(I)$. According to \cite[V.3.1.11]{Bla06}, a properly infinite C*-algebra has stable rank $\infty$. Moreover, $C^*(\A)$ is generated by projections because it is generated by the $W^*$-algebras $\iota_I(\A(I))$, which of course are generated by projections. 

By Corollary \ref{corCompacts}, $\pi_0(C^*(\A))$ contains $\K(\H)$, so that $C^*(\A)$ contains the nonzero closed ideal 
$\pi_0^{-1}(\K(\H))$.  This ideal is not equal to  $C^*(\A)$ e.g. because $\K(\H)$ has a trivial intersection with any of the type III factors 
$\A(I)=\pi_0(\iota(\A(I)))$. Hence $C^*(\A)$ is not simple. As another consequence,  the finite projections in $\K(\H)$ are finite projections in $\pi_0(C^*(\A))$ so that the latter is not purely infinite, \cite[V.2.2.24]{Bla06}. Hence, by \cite[V.2.2.22]{Bla06}, $C^*(\A)$ is not purely infinite. 

Finally, suppose $C^*(\A)$ were exact. Let $I$ be an interval in $\I$ and let $\A(I)$ be the corresponding type III factor. Then 
$\A(I)$, being a purely infinite von Neumann algebra, contains a countably decomposable infinite-dimensional type I factor $F$, see 
\cite[III.1.5.6 (ii)]{Bla06} (we are not assuming the split property). Since exactness is hereditary to subalgebras, $\iota_I(F)$ and hence
$F$ would be exact. But type I$_\infty$ factors are not exact \cite[II.9.6.6]{Bla06}, a contradiction. So $C^*(\A)$ is not exact.
\end{proof}
\begin{remark}\label{remarkCenter} For many interesting conformal nets $\A$ on $S^1$ one can use \cite[Proposition 5.1.]{FRS2} to show that $C^*(\A)$ has nontrivial center.
\end{remark}

Another  C*-algebra which is useful in the study of a conformal net $\A$ on $S^1$ is obtained as the quasi-local  
C*-algebra 
 $C^*(\A|_\R)$ of the restriction of $\A$ to $\R \simeq S^1\setminus \{-1\}$. More precisely let $\I_\R=\{I \in \I: -1 \notin \overline{I} \}$. Then $\I_\R$ is directed under inclusion and $C^*(\A|_\R)$ is defined as the norm closure of the unital $*$-algebra $\bigcup_{I \in \I_\R}\A(I)$. As shown by the following proposition the quasi-local  C*-algebras $C^*(\A|_\R)$ share many properties with the universal  C*-algebras $C^*(\A)$ but there are also various important differences. 
\begin{proposition}\label{PropQuasilocal}
Let $\A$ be a conformal net on $S^1$. Then the quasi-local  C*-algebra $C^*(\A|_\R)$ is a simple unital, properly infinite and purely infinite  C*-algebra, which has stable rank $\infty$, real rank $0$, it is generated by projections and has trivial $K$-theory. It is neither separable nor exact. If $\A_1$ and $\A_2$ are conformal nets on $S^1$ with the split property then $C^*(\A_1 |_\R)$ is isomorphic to 
$C^*(\A_2 |_\R)$. 
\end{proposition}

\begin{proof}
The proof follows from the fact that $C^*(\A|_{\R})$ is the direct limit of countably decomposable type III factors, which are simple, nonseparable properly infinite and purely infinite   C*-algebras \cite[III.1.7.11\&p.431]{Bla06}.  Moreover, type III factors are not exact being properly infinite von Neumann algebras (\cf the proof of Proposition \ref{prop:U-propsCA} above). The K-groups of a direct limit are the direct limit of the K-groups, which are 0 in the present case \cite[V.1.1.16]{Bla06}. The pure infiniteness follows by the same reasoning, using \cite[V.2.2.26]{Bla06}. Real rank is stable under direct limits, and every von Neumann algebra has real rank 0, \cite[V.3.2.10-11]{Bla06}. Now, let $\A_1$ be a conformal net on $S^1$ with the split property. Then there is an increasing sequence  of countably decomposable infinite-dimensional type I factors $F_1 \subset F_2 \subset \dots \subset F_n \subset \dots$ such that $F'_n \cap F_{n+1}$ is infinite-dimensional for all positive integers $n$ and such that for every $I\in \I_\R$ there exists a positive integer $m$ such that  $\A_1(I) \subset F_m$. Accordingly 
$\bigcup_{I \in \I_\R}\A_1(I) = \bigcup_{n\in \N}F_n$ and hence $C^*(\A_1 |_\R)$ is equal to the norm closure 
$B_1$ of $\bigcup_{n\in \N}F_n$. 
If $\H$ is a separable Hilbert space then the inclusion map $\phi_1 : B(\H) \ra B(\H) \bar{\otimes} B(\H) $ defined by 
$\phi_1(x)=x\otimes \unit$ gives rise to a directed system 
$B(\H) \overset{\phi_1}{\ra}  B(\H) \bar{\otimes} B(\H) \overset{\phi_2}{\ra} B(\H) \bar{\otimes} B(\H)\bar{\otimes} B(\H) \dots$ 
and it is fairly easy to see that the corresponding  C*- inductive limit $B$ is isomorphic to $B_1$ and hence to 
$C^*(\A_1 |_\R)$. If $\A_2$ is another conformal net on $S^1$ with the split property then in the same way we see that 
$C^*(\A_2 |_\R)$ is isomorphic to $B$ and consequently it is isomorphic to $C^*(\A_1 |_\R)$. 
Note that incidentally we have shown that the above  C*-inductive limit $B$ of infinite-dimensional countably decomposable type I factors has all the properties in the proposition. In particular it is a simple and purely infinite  
C*-algebra. 
\end{proof}

The properties in Proposition \ref{prop:U-propsCA} are far from being enough to determine the possible isomorphism classes of the 
universal  C*-algebras of  conformal nets even if we restrict to the very special (but important) class of completely rational nets. 
However, as we shall see in the rest of this section we can get much more information for their images $\pi(C^*(\A))$ in locally normal representations $\pi$, especially in the cases where the representations have finite statistical dimension.
We begin with an abstract result on intermediate $*$-subalgebras of inclusions of von Neumann algebras having finite index in the sense of Pimsner and Popa \cite{PiPo86}, see  \cite[Definition 1.1.1]{Popa95}, which will play a crucial role in our analysis. 
 
\begin{proposition}\label{intermediatesubalgebras} 
Let  $N \subset M$ be a finite index inclusion of countably decomposable von Neumann algebras with $N$ a type {\rm III} factor. Then every $*$-subalgebra $B_0$ of $M$ containing $N$ is $\sigma$-weakly closed. 
\end{proposition}

\begin{proof}
Let $E: M\ra N$ be a faithful normal conditional expectation with finite index and let $B$ be the norm closure of $B_0$. 
Then the restriction of $E$ to $B$ gives a conditional expectation of the  C*-algebra $B$ onto $N$ with finite index in the sense of 
\cite[Def. 2.1]{Izumi02}. Now, by \cite[III.1.7.11]{Bla06}, $N$ is a simple  C*-algebra. Hence, by 
\cite[Corollary 3.4]{Izumi02} 
there exist elements $w_1,\dots w_n$ in $B$ such that $$b = \sum_{i=1}^n w_i E(w_i^*b)$$ for all $b \in B$. Now, if $b$ is in the $\sigma$-weak closure of $B$ in $M$ then, recalling that $E$ is normal we find again $$b = \sum_{i=1}^n w_i E(w_i^*b) \in B$$ and hence $B$ is $\sigma$-weakly closed. By \cite[Theorem 1.1.6 b)]{Popa95} there exists $w\in B$ such that $E(w^*w)= \unit$ and $b=wE(w^*b)$ for all 
$b\in B$ so that $B=wN$ . Now let $N_0= \{n \in N: w n \in B_0 \}$. Then $N_0$ is a right ideal of $N$ such that $wN_0 = B_0$. Hence 
$N_0 = E(w^*wN_0) = E(w^*B_0)$ is norm dense in $E(w^*B)=N$. It follows that $N_0$ must contain some element from the open 
set of invertibles in $N$ and hence, being a right ideal,  it must coincide with $N$.   Accordingly $B_0=B$ and the conclusion follows. 
\end{proof}  

We now come back to the universal  C*-algebras of conformal nets on $S^1$.  

\begin{proposition}\label{typeIgenerate} 
Let $\A$ be a conformal net on $S^1$ with the split property and let $F$ be a type {\rm I} factor such that 
$\A(I_1) \subset F \subset  \A(I_2)$ for a given pair $I_1, I_2 \in \I$ with $\overline{I_1} \subset I_2$.  Then 
$\pi(C^*(\A))''= \pi_{I_2}(F)\vee \left(\pi_{I_2}(F)'\cap \pi(\Alg(\A)) \right)''$ for every locally normal representation $\pi$ of $\A$.    
 \end{proposition} 
 \begin{proof}  Let $\pi$ be a locally normal representation of $\A$, then  $\pi_{I_2}(F)$ is an infinite-dimensional countably decomposable type I factor and hence is isomorphic to $B(\H)$ where $\H$ is the  (separable) vacuum Hilbert space of 
 $\A$. Accordingly there exists a Hilbert space $\tilde{\H}$ and a unitary  $w: \H_\pi \mapsto \H \otimes \tilde{\H}$ such that $w\pi_{I_2}(F)w^*= B(\H) \otimes \unit$.  Let $I \in \I$ be any interval containing $I_2$, so that $F \subset \A(I)$ and 
 $\pi_{I_2}(F)=\pi_I(F)$. Since $\pi_I(\A(I))' \subset \pi_I(F)'$, $w\pi_I(\A(I))' w^*= \unit \otimes R$ for some von Neumann algebra $R$ on $\tilde{\H}$.  Accordingly, 
 \begin{eqnarray*}
 \pi_I(F)' \cap \left(\pi_I(F)\vee \pi_I(\A(I))'  \right) &=& w^* \left(\unit \otimes B(\tilde{\H}) \right)\cap \left( B(\H) \bar{\otimes} R \right)w \\
 &=& w^* \left(\unit \otimes R \right) w  = \pi_I(\A(I))' 
 \end{eqnarray*} 
 and by taking the commutants in the above equality we find 
 $$\pi_I(\A(I))= \pi_I(F) \vee \left( \pi_I(F)' \cap \pi_I(\A(I)) \right) \subset \pi_I(F) \vee \left( \pi_I(F)' \cap \pi(\Alg(\A)) \right)''. $$ 
 By additivity of the net $\A$ and the local normality of $\pi$ the algebra generated by the von Neumann factors $\pi_I(\A(I))$ with $I\in \I$ containing $I_2$ is strongly operator dense in $\pi(C^*(\A))''$. Hence 
$$\pi(C^*(\A))'' \subset \pi_{I_2}(F) \vee \left( \pi_{I_2}(F)' \cap \pi(\Alg(\A)) \right)''. $$
The remaining inclusion is obvious.
\end{proof}

\begin{lemma}\label{typeIIIrelativecommutant} Let $\A$ be a conformal net on $S^1$ and let $\pi$ be a locally normal representations of $\A$ with finite statistical dimension. Then the $*$-algebra $\pi_I(\A(I))'\cap \pi(\Alg(\A))$ is $\sigma$-weakly closed in $B(\H_\pi)$ for all $I\in \I$. 
\end{lemma}
\begin{proof} Let $\pi$ be a locally normal representation of $\pi$ with statistical dimension $d(\pi) < +\infty$. Then, for any $I\in \I$, 
we have
$$\pi_{I'}(\A(I')) \subset \pi_I(\A(I))'\cap \pi(\Alg(\A)) \subset \pi_I(\A(I))' .$$ 
The inclusion of countably decomposable type III factors $\pi_{I'}(\A(I')) \subset \pi_I(\A(I))'$ has index 
$d(\pi)^2 < +\infty$ and the conclusion follows from Proposition \ref{intermediatesubalgebras}.
\end{proof}

\begin{proposition}\label{typeIrelativecommutant}
 Let $\A$ be a conformal net on $S^1$ with the split property and let $F$ be a type {\rm I} factor such that 
$\A(I_1) \subset F \subset  \A(I_2)$ for a given pair $I_1, I_2 \in \I$ with $\overline{I_1} \subset I_2$. Then the $*$-algebra 
$\pi_{I_2}(F)'\cap \pi(\Alg(\A))$ is $\sigma$-weakly closed in $B(\H_\pi)$, for every locally normal representation $\pi$ of $\A$ with finite statistical dimension.
\end{proposition}

\begin{proof}  Let $\pi$ be a locally normal representation of $\pi$ with statistical dimension $d(\pi) < +\infty$. Since $\pi_{I_2}(F)'\cap \pi(\Alg(\A)) \subset \pi_{I_1}(\A(I_1))'\cap \pi(\Alg(\A))$ we have 
$$\pi_{I_2}(F)'\cap \pi(\Alg(\A)) = \pi_{I_2}(F)'\cap \pi_{I_1}(\A(I_1))' \cap \pi(\Alg(\A))$$ and the conclusion follows from Lemma 
\ref{typeIIIrelativecommutant}.
\end{proof}

\begin{corollary}\label{irrF} 
 Let $\A$ be a conformal net on $S^1$ with the split property and let $F$ be a type {\rm I} factor such that $\A(I_1) \subset F \subset  \A(I_2)$ for a given pair $I_1, I_2 \in \I$ such that $\overline{I_1} \subset I_2$. Then $\pi_{I_2}(F)' \subset \pi(\Alg(\A))$ for every irreducible locally normal representation $\pi$ of $\A$ with finite statistical dimension.
\end{corollary}

\begin{proof}  Let $\pi$ be an irreducible locally normal representation of $\pi$ with statistical dimension $d(\pi) < +\infty$. By Proposition \ref{typeIrelativecommutant} $\pi_{I_2}(F)'\cap \pi(\Alg(\A))$ is a von Neumann algebra on $B(\H_\pi)$ and by 
Proposition \ref{typeIgenerate} and the irreduciblity of $\pi$ we have $\left(\pi_{I_2}(F)'\cap \pi(\Alg(\A)) \right)\vee \pi_{I_2}(F) =\B(\H_\pi)$. 
As  $\pi_{I_2}(F)$ is a type I factor, we must therefore have $\pi_{I_2}(F)'\cap \pi(\Alg(\A))=\pi_{I_2}(F)'$ i.e. 
$\pi_{I_2}(F)' \subset \pi(\Alg(\A))$
\end{proof}

\begin{theorem}\label{piC*A_irr} Let $\A$ be a conformal net on  $S^1$ with the split property. Then $\pi(\Alg(\A))=\pi(C^*(\A))=B(\H_\pi)$, for every irreducible locally normal representation $\pi$ of $\A$ with finite statistical dimension. 
\end{theorem}

\begin{proof}
Let $\pi$ be any irreducible locally normal representation of $\A$ with finite statistical dimension. It is enough to prove that 
$\pi(\Alg(\A)) = B(\H_\pi)$
Let $I_1, I_2, I_3 \in \I$ be such that $\overline{I_1} \subset I_2$ and $\overline{I_2} \subset I_3$ and let $F_1, F_2$ be type I factors 
such that $\A(I_1) \subset F_1 \subset \A(I_2) \subset F_2 \subset \A(I_3)$. Then $\pi_{I_3}(F_1)$ and  $\pi_{I_3}(F_2)$ are 
infinite-dimensional factors with infinite-dimensional commutants on the separable Hilbert space $\H_\pi$. Moreover the type I factor 
$\pi_{I_3}(F_1)' \cap \pi_{I_3}(F_2)$ is infinite-dimensional since it contains the type III factor $\pi_{I_3}(\A(I_2))' \cap \pi_{I_3}(F_2)$.
Clearly $\pi_{I_3}(F_2) \subset \pi(\Alg(\A))$ and by Corollary \ref{irrF} we also have $\pi_{I_3}(F_1)' \subset \pi(\Alg(\A))$. 
Let $p$ be a minimal projection in $\pi_{I_3}(F_1)$. Then $p$ is equivalent to the identity in $\pi_{I_3}(F_2)$. 
Let $w \in \pi_{I_3}(F_2)$ be an isometry such that $ww^*=p$. Then $wB(\H_\pi)w^*=pB(\H_\pi)p=p\pi_{I_3}(F_1)'$ and hence 
$B(\H_\pi)=w^*p\pi_{I_3}(F_1)'w = w^*\pi_{I_3}(F_1)'w \subset \pi(\Alg(\A)) . $
\end{proof}
In order to understand the reducible case we need to consider the center of $\pi(C^*(\A))$.

\begin{proposition}\label{piC*A_center}
Let $\A$ be a conformal net on $S^1$ with the split property. Then we have 
$\ZZ(\pi(\Alg(\A)))= \ZZ(\pi(C^*(\A)))= \ZZ(\pi(C^*(\A))'')$, for every locally normal representation $\pi$ of $\A$ with finite statistical dimension. 
\end{proposition}

\begin{proof} Let $\pi$ be a locally normal representation of $\A$ with finite statistical dimension. 
Then $\pi$ is a direct sum of finitely many irreducible and hence cyclic representations of $\A$ so that $\H_\pi$ is separable. 
 Since $\pi(\Alg(\A))'= \pi(C^*(\A))'$ we have only to show that $\ZZ(\pi(C^*(\A))'') \subset \pi(\Alg(\A))$.  
Let $I_1, I_2 \in \I$ be such that $\overline{I_1} \subset I_2$, and let $F$ be a type I factor such that $\A(I_1) \subset F \subset \A(I_2)$. By Proposition \ref{typeIrelativecommutant}, $R:=\pi_{I_2}(F)' \cap \pi(\Alg(\A))$ is a von Neumann algebra on $\H_\pi$. Moreover by Proposition \ref{typeIgenerate}, $\pi_{I_2}(F) \vee R = \pi(C^*(\A))''$. 
Now let $w: \H_\pi \ra \H_\pi \otimes \H_\pi$ be a unitary operator such that $w\pi_{I_2}(F)w^*=B(\H_\pi)\otimes \unit$. Then $wRw^* = \unit \otimes \tilde{R}$ for some von Neumann algebra $\tilde{R}$ on $\H_\pi$ and accordingly 
$w\pi(C^*(\A))''w^*=B(\H_\pi)\bar{\otimes} \tilde{R}$ so that 
$\ZZ(\pi(C^*(\A))'')=w^*\left( 1 \otimes \ZZ(\tilde{R}) \right)w = \ZZ(R) \subset \pi(\Alg(\A)) .$
\end{proof}

We are now ready to prove one of the main results of this paper. 

\begin{theorem}\label{piC*A_reducible} 
Let $\A$ be a conformal net on $S^1$ with the split property. Then $\pi(\Alg(\A))= \pi(C^*(\A)) =\pi(C^*(\A))''$ for every locally normal representation $\pi$ of $\A$ with finite statistical dimension. 
\end{theorem}

\begin{proof} Let $\pi$ be a locally normal representation of $\A$ with finite statistical dimension. Then $\pi$ is a direct sum of finitely many irreducible representations and hence $\ZZ(\pi(C^*(A))'')$ is finite-dimensional. 
We only have to show that $\pi(\Alg(\A)) = \pi(C^*(\A))''$

Let $e_1, e_2, \dots e_n$ be a maximal orthogonal family of minimal projections in $\ZZ(\pi(C^*(A))'')$. Then 
$$\pi(C^*(\A))'' = \bigoplus_{i=1}^n e_i\pi(C^*(\A))'' e_i .$$
 The representation $e_i\pi(\cdot) e_i$ on $e_i\H_\pi$ is a  finite multiple of an irreducible representation with finite statistical dimension, for $i=1,2, \dots, n$. Hence, by Theorem \ref{piC*A_irr}, we have 
$e_i\pi(C^*(\A))'' e_i=e_i\pi(\Alg(\A))e_i$, for $i=1,2, \dots, n$. By Proposition \ref{piC*A_center} we have $e_i \in \pi(\Alg(\A))$, for 
$i=1,2, \dots, n$. It follows that   
$$\pi(C^*(\A))'' = \bigoplus_{i=1}^n e_i\pi(\Alg(\A))e_i = \pi(\Alg(\A)) .$$
\end{proof}

\begin{corollary}\label{piC*A_sum} Let $\A$ be a conformal net on $S^1$ with the split property. Then $\pi(C^*(\A))$ is a direct sum of finitely many countably decomposable infinite-dimensional type {\rm I} factors for every locally normal representation $\pi$ of $\A$ with finite statistical dimension. 
\end{corollary}

For completely rational nets the assumption of finiteness of the statistical dimension in  Theorem  \ref{piC*A_reducible}
and Corollary  \ref{piC*A_sum}  turns out to be unnecessary as  shown by the following theorem.

\begin{theorem}\label{piC*A_rational} 
Let $\A$ be a completely rational conformal net on $S^1$ and let $\pi$ be a locally normal representation of $\A$. 
Then $\pi(\Alg(\A))= \pi(C^*(\A)) =\pi(C^*(\A))''$ and moreover $\pi(C^*(\A))$ is a direct sum of finitely many countably decomposable 
infinite-dimensional type {\rm I} factors. 
\end{theorem}
\begin{proof} Let $\pi$ be a locally normal representation of the completely rational conformal net $\A$. Then according to \cite[Cor.39]{KLM} it is equivalent to a direct sum of irreducible locally normal representations with finite  statistical dimension. Since there are only a finite number of equivalence classes of  these representations it follows that  $\pi(C^*(\A))''$ is a direct sum of finitely many 
type I factors which are countably decomposable and infinite-dimensional because irreducible locally normal representations of 
a conformal net $\A$ acts on infinite-dimensional separable Hilbert spaces. Moreover, $\pi$ is quasi-equivalent to a locally normal representation with finite index and hence $\pi(\Alg(\A))= \pi(C^*(\A)) =\pi(C^*(\A))''$ by Theorem 
\ref{piC*A_reducible}.
\end{proof} 

If $\A$ is a completely rational conformal net on $S^1$ then Theorem \ref{piC*A_rational} gives a complete description of $\pi(C^*(\A))$ for all representations which are relevant for CFT and hence in view of the applications to quantum field theory the result appears to be completely satisfactory. If $\A$ is a conformal net which is not completely rational then in order to have 
a complete description, we can apply Theorem \ref{piC*A_reducible} if we assume the split property and if we restrict ourselves to representations 
which are quasi-equivalent to locally normal representations with finite statistical dimension. 
The split property does not appear to be a restrictive assumption. As far as we know there is no conformal net for which it has been proved that the split property fails. Moreover the standard construction of net without the split property by means of infinite tensor products gives rise 
to non-diffeomorphism covariant nets as shown in \cite[Section 6]{CW05}. On the other hand, although locally normal irreducible representations with infinite statistical dimension are in a certain sense pathological they naturally appear in  non-rational conformal net models \cite{Fre95,Car03,Car04,LX}. For this reason it could be useful to have some information about the structure of $\pi(C^*(\A))$ also in the case in which the locally normal representation $\pi$ is not assumed to have finite statistical dimension and 
in the remaining part of this section we will tackle this problem. This will also lead us to Corollary \ref{corCompacts} which was used 
in the proof of Proposition \ref{prop:U-propsCA} in order to show that the universal  C*-algebra $C^*(\A)$ of a conformal net on $S^1$ is neither simple nor purely infinite. 

  First, we shall show that $\pi(C^*(\A))$ contains ``heat kernels". We write $L_0^\pi$ for the conformal Hamiltonian (the infinitesimal generator of the rotation subgroup of $\Diff$ as explained in Section \ref{sec:prelim}) of $\A$ in the locally normal representation $\pi$. 
Since for $t\in \R$ $$\rme^{\rmi t L_0^\pi} \in \bigvee_{I\in \I} \pi_I(A(I))= \pi(C^*(\A))'' ,$$ 
recalling that $L_0^\pi \geq 0$, we also have $\rme^{-t L_0^\pi} \in \pi(C^*(\A))''$ for all $t \geq 0$.  Hence, as a consequence of 
Theorem \ref{piC*A_reducible}, if $\pi$ has finite statistical dimension then $\rme^{-t L_0^\pi} \in \pi(\Alg(\A))$ for all $t\geq 0$.

If $\pi$ is irreducible but is not assumed to have finite  statistical dimension then it follows from \cite{DFK} that, for any $t\in \R$, $\rme^{\rmi t L_0^\pi}$ is a finite product of local unitaries and therefore lies in $\pi(\Alg(\A))$. This result extends in the following way (the split property is not assumed):

\begin{theorem} \label{th:univ-heat-inCA}
Let $\A$ be a conformal net on $S^1$, $\pi$ a (not necessarily irreducible) locally normal representation of $\A$. Then $\rme^{-t L_0^\pi} \in \pi(\Alg(\A))$, for every $t\ge0$.
\end{theorem}
 
\begin{proof}
 
There are two important ingredients for our proof: one of them can be 
found in \cite{BDL}, the other in \cite{Wei06}. To keep the 
argument relatively short, we will extensively refer to those two
papers. At the same moment, however, we shall also try to avoid being ``unreadable'' and recall at 
least the basic concepts and ideas.
 
Consider the group $\PST \subset \Diff$ defined as the set of 
diffeomorphisms $g\in \Diff$ for which there exists a 
Möbius transformation $h\in \PSL$ so that
\[
g(z)^2 = h(z^2), \quad z\in\S.
\]
It is the second cover of $\PSL$, and as such, it is isomorphic to ${\rm SL}(2,\R)$. Through a 
lifting procedure, to each one-parameter group of $\PSL$ one can define 
the corresponding one-parameter group of $\PSL^{(2)}$. Thus corresponding 
to the usually introduced {\it translations} $x\mapsto \tau_x$ and {\it 
rotations} $\alpha\mapsto R_\alpha$, one has the one-parameter groups 
$x\mapsto \tau^{(2)}_x \in \PST$ and $\alpha\mapsto R^{(2)}_\alpha \in 
\PST$; see further details in \cite{Wei06}. Of course, for rotations
one has the simple relation $R^{(2)}_\alpha = R_{\alpha/2}$.
 
Let $S^1_\pm=\{z\in S^1: \, \pm{\rm Im}(z)>0\}$ be the upper and lower 
half-circles, respectively. A ``$2$-translation'' $\tau^{(2)}_x$ leaves 
the points $\pm 1\in S^1$ fixed and hence can be ``cut into two'': it 
can be written as the composition of 
two homeomorphisms localized in the upper and the lower half-circle, respectively. 
However, these homeomorphisms are not smooth and hence one cannot 
substitute them in the positive energy representation $U$ given 
with the conformal net $\A$. Nevertheless, in \cite{Wei06} it was proved that
there exist two self-adjoint operators $K_\pm $ affiliated with 
$\A(S^1_\pm)$ respectively, and bounded from below, such that 
\begin{equation}\label{folbontas}
U(\tau^{(2)}_x) = \rme^{ix K_+}\,  \rme^{ix K_-}, \quad x\in \R.
\end{equation}
 
For every locally normal representation $\pi$ of $\A$, there exist two unique 
strongly continuous projective representations $U_\pi^{(1)} := U_\pi$ 
and $U_\pi^{(2)}$ of $\PSL$ and $\PST$, 
respectively, that $U^{(n)}_\pi(\PSL^{(n)})\subset \pi(C^*(\A))''$ and 
\[
\Ad(U^{(n)}_\pi(g))\circ \pi_I = \pi_{gI}\circ
{\rm Ad}(U(g)),
\]
for $n=1,2$, $g\in \PSL^{(n)}$ and intervals $I\subset S^1$. 
In fact in \cite[Prop.3.4]{Wei06}, apart from existence, it was also shown that 
the representations $U^{(n)}_\pi$ are of positive energy type
and that
\[
U^{(2)}_\pi(R^{(2)}_{\alpha/2})=
U^{(2)}_\pi(R_{\alpha})= U_\pi(R_\alpha),
\]
and
\begin{equation}\label{eq:univ-heat2}
U^{(2)}_\pi(\tau^{2}_x) = \pi_{S^1_+}(\rme^{ix K_+})\pi_{S^1_-} (\rme^{ix K_-}),
\end{equation}
for all $x,\alpha \in \R$. Here of course the equations are meant in 
the projective sense, \ie up to a phase. Note that, by construction, \eqref{eq:univ-heat2} is a product of local elements, hence in 
$\pi(C^*(\A))$.
 
The projective representations $U^{(n)}_\pi$ lift to strongly 
continuous positive energy representations of the universal covering of 
$\PSL$. It is exactly through this lift that the self-adjoint 
generator of rotations $L^\pi_0\geq 0$ is defined and we have that 
\[
\rme^{\rmi \alpha L^\pi_0} = U^{(n)}_\pi(R_\alpha), \quad \alpha\in \R,\; n=1,2, 
\]
in the projective sense.

On the other hand, \cite[Th.3.3]{BDL} applied to the representation $U_\pi^{(2)}$ of $\PST$ provides us with a 
decomposition of $\rme^{-t L_0^\pi}$ into the product of $3$ operators;
using \eqref{eq:univ-heat2}, by which we can further split each 
of those $3$ operators into a product of two, we finally get that,
for all $t>0$,
\begin{align}
\nonumber
\rme^{-t L_0^\pi} = & c_t
\pi_{S^1_+}(\rme^{-{\rm tanh}(t) K_+}) \,
\pi_{S^1_-}(\rme^{-{\rm tanh}(t) K_-}) \, 
\pi_{\rmi S^1_+}(\rme^{-{\rm sinh}(t) \tilde{K}_+}) \times
\\
&\times
\pi_{\rmi S^1_-}(\rme^{-{\rm sinh}(t) \tilde{K}_-})\,
\pi_{S^1_+}(\rme^{-{\rm tanh}(t)K_+}) \,
\pi_{S^1_-}(\rme^{-{\rm tanh}(t)K_-}), 
\end{align}
where $\tilde{K}_\pm = U(R_{\pi/2}) K_\pm U(R_{\pi/2})^*$ are positive 
operators affiliated to $\A(\rmi S^1_\pm)$, the algebras over the left and right half-circle, and the scalar $c_t\in\R_+$ depends only on $\pi$ and $t$. 
\end{proof}   

If $\A$ is a conformal net on $S^1$ with the split property and $\pi$ is a locally normal representation of $\A$ with finite statistical dimension, then by Theorem \ref{piC*A_irr}, $\pi(C^*(\A))=B(\H_\pi)$ so that in particular $\K(\H_\pi) \subset \pi(C^*(\A))$. 
In the case of infinite statistical dimension we are not able to prove that $\pi(C^*(\A))=B(\H_\pi)$ but thanks to Theorem 
\ref{th:univ-heat-inCA} we can still prove that $\K(\H_\pi) \subset \pi(C^*(\A))$ under the mild assumption that $L_0^\pi$ has a 
finite-dimensional eigenspace. Note that, to the best of our knowledge, $L_0^\pi$ has no  infinite-dimensional eigenspace in all studied examples with diffeomorphism symmetry. 

\begin{proposition}\label{propCompacts}
Let $\A$ be a conformal net on $S^1$ and let $\pi$ be an irreducible locally normal representation of $\A$.  Assume that 
$L_0^\pi$ has a finite-dimensional eigenspace. Then $\K(\H_\pi) \subset \pi(C^*(\A))$.  
\end{proposition}

\begin{proof} 
Let $\pi$ be an irreducible locally normal representation of $\A$. Since $\rme^{\rmi 2\pi L_0^\pi}\in \pi(C^*(\A))'$
 it must be a multiple of the identity and hence the spectrum of $L_0^\pi$ is contained in $h_\pi + \NN$ where $h_\pi$ is the lowest 
 eigenvalue of $L_0^\pi$. Thus for every $\lambda>0$, the characteristic function $\chi_{\{\lambda \} }$ is continuous on the spectrum
$\sigma(\rme^{- L_0^\pi})$ of $\rme^{-t L_0^\pi}$ since the latter is discrete on $(0,1)$. By Theorem 
\ref{th:univ-heat-inCA}, $\rme^{-L_0^\pi}\in \pi(C^*(\A))$ Accordingly, if $L_0^\pi$ has a finite-dimensional eigenspace corresponding to the eigenvalue $\lambda \geq 0$, then the projection 
$\chi_{\{e^{-\lambda}\}}(\rme^{- L_0^\pi})$ is a nonzero compact operator in $\pi(C^*(\A))$. Hence $\pi(C^*(\A))$ contains all compact operators, see e.g. \cite[IV.1.2.5]{Bla06}.
\end{proof}
\begin{corollary}\label{corCompacts}
Let $\A$ be a conformal net on $S^1$ and let $\pi_0$ be the vacuum representation of $\A$ on the Hilbert space $\H$. Then $\K(\H) \subset \pi_0(C^*(\A))$.  
\end{corollary}

\begin{proof} 
$\pi_0$ is irreducible and locally normal. Moreover by the definition of conformal nets on $S^1$ the eigenspace of 
$L_0^{\pi_0}=L_0$ corresponding to the eigenvalue $0$ is one-dimensional and the conclusion follows from 
Proposition \ref{propCompacts}.
\end{proof}

\section{The locally normal universal C*-algebra }\label{sec:loc-normal}
 
As already pointed out in the introduction the non-locally normal representations do not apperar to have direct rilevance for CFT . Therefore, it seems natural to ask for a more manageable C*-algebra than 
$C^*(\A)$ taking account only of the locally normal representations. In this section we define the most natural candidate for this purpose and analyze some of its properties. 
 
\begin{definition}\label{def:LN-univC}
The \emph{locally normal universal representation} $(\pi_{\locn},\H_{\locn})$ of $C^*(\A)$ is the direct sum of GNS representations $\pi_\varphi$  over all states $\varphi$ of $C^*(\A)$ for which $\pi_\varphi\circ\iota_I$ is a normal representation of $\A(I)$, for all $I\in\I$. It defines the \emph{locally normal universal C*-algebra} $C^*_{\locn}(\A):=\pi_{\locn}(C^*(\A))$ of $\A$.
\end{definition}
 
\begin{remark}\label{rem:LN-univCA}
\begin{itemize}
  \item[(1)] Clearly, $C^*_{\locn}(\A)\simeq C^*(\A)/\ker(\pi_{\locn})$. Consequently, $C^*_{\locn}(\A)$ satisfies the following 
  \emph{universal property} for locally normal representations: for every locally normal representation $\pi$ of $\A$ on $\H_\pi$, there is a unique representation $\hat{\pi}:C^*_{\locn}(\A)\ra B(\H_\pi)$ such that
\[
\pi_I = \hat{\pi}\circ \iota_I,\quad I\in \I,
\]
where, with a slight abuse of notation, we have also denoted by $\iota_I$ the natural inclusion of $\A(I)$ in $C^*_{\locn}(\A)$. 
Accordingly, for most purposes, if one deals only with locally normal representations
one can replace $C^*(\A)$ with $C^*_{\locn}(\A)$.
As before then, we shall drop the symbol `` $\hat{\cdot}$ " and identify $\A(I)$ with its image in $C^*_{\locn}(\A)$.
\item[(2)] If $\A$ is completely rational, then according to \cite[Cor.39]{KLM} every locally normal representation of $\A$ is equivalent to a direct sum of irreducible locally normal representations with finite  statistical dimension. Considering then the \emph{reduced locally normal universal representation} $(\pi_{\red},\H_{\red})$ of $C^*(\A)$, namely the direct sum of a maximal family of mutually inequivalent irreducible locally normal representations of $C^*(\A)$, we see that $\pi_{\red}$ and $\pi_{\locn}$ are
quasi-equivalent so that

\[
  \pi_{\red}(C^*(\A)) \simeq \pi_{\locn}(C^*(\A)).
\]
Moreover, every locally normal representation of $C^*(\A)$ is quasi-equivalent to a subrepresentation of $\pi_{\red}$.
The reason of considering this reduced representation, in which we get rid of an infinite multiplicity, shall become clear in Theorem \ref{th:univ-K-inCA}. We write $C^*_{\red}(\A) := \pi_{\red}(C^*(\A))$.

\end{itemize}
\end{remark}

Based on the new concepts introduced in the present setting, we can reformulate the results from Section 
\ref{sec:universal} as follows:
 
\begin{theorem}\label{th:LN-main} Let $\A$ be a conformal net on $S^1$. Then the following hold. 
\begin{itemize}
\item[$(1)$] $C^*_{\locn}(\A)$ has all the properties of $C^*(\A)$ in Proposition \ref{prop:U-propsCA}
\item[$(2)$] The elements $\rme^{-t L_0^{\pi_{\locn}}}$, with $t>0$, lie in $C^*_{\locn}(\A)$.
\item[$(3)$] Suppose $\A$ is completely rational, with $n$ sectors. Then $C^*_{\locn}(\A)$ is weakly closed,
\[
C^*_{\locn}(\A) \simeq B(\H)^{\oplus n}, \quad \ZZ(C^*_{\locn}(\A)) \simeq \C^n.
\]
\end{itemize}
\end{theorem}
 
\begin{proof}
(1) follows by the same arguments given in the proof of Proposition \ref{prop:U-propsCA}.
 
(2) As a sum of locally normal representation, $\pi_{\locn}$ is locally normal again, so the statement follows from Theorem 
\ref{th:univ-heat-inCA}.
 
(3) From the definition of $C^*_{\locn}(\A)$ and Theorem \ref{piC*A_rational} we know that $C^*_{\locn}(\A)$ is weakly closed and that
\[
C^*_{\locn}(\A) \simeq C^*_{\red}(\A) \simeq \Big(\bigoplus_{[\pi]\in[\Delta_{\irr}]}\pi(C^*(\A)) \Big) 
\simeq B(\H)^{\oplus n}.
\]
\end{proof}

\begin{proposition}\label{th:univ-K-inCA}
Let $\A$ be a completely rational conformal net on $S^1$ with $n$ sectors. Then the following hold.
\begin{itemize}
\item[$(1)$] Every irreducible locally normal representation of $C^*(\A)$ contains the compact operators. The restriction of representations of $C^*_{\locn}(\A) \simeq C^*_{\red}(\A)$
to the subalgebra of compacts 
\[
  \KK:=\K(\H_{\red})\cap C^*_{\red}(\A)
\]
gives rise to an isomorphism from the category of locally normal (unital) representations of $C^*(\A)$ onto the category of nondegenerate representations of $\KK$. In particular there is a one-to one correspondence between 
equivalence classes of nondegenerate irreducible representations of $\KK$ and sectors of $\A$.
\item[$(2)$] For every localized endomorphism $\rho$ of $C^*(\A)$, there exists a unique normal faithful endomorphism $\hat{\rho}$ of  $C^*_{\red}(\A)$ such that $\hat{\rho}(\pi_{\red}(x)) = \pi_{\red}(\rho(x))$
for all $x \in C^*(\A)$. Moreover, if $\rho$ has finite statistical dimension then $\hat{\rho}(\KK) \subset \KK$. 
\end{itemize}
\end{proposition}
 
\begin{proof}
(1) From Theorem \ref{th:LN-main} (2), we have $\KK\simeq \K^{\oplus n}$ since $\H_{\red}$ is a finite direct sum of separable Hilbert spaces. As a consequence, every nondegenerate representation of $\KK$ has a (unique) normal extension to $\KK''= C^*_{\red}(\A) \simeq B(\H)^{ \oplus n}$, and the claim follows.
 
(2) The locally normal representation $\pi_{\red}\circ \rho$ is quasi-equivalent to a subrepresentation of $\pi_{\red}$. Accordingly, there is a unique normal representation $\hat{\rho}$ (corresponding through the quasi-equivalence to the restriction to the subrepresentation subspace) of $C^*_{\red}(\A)$ on $\H_{\red}$ 
satisfying 
\[
 \hat{\rho}(\pi_{\red}(x))= \pi_{\red}(\rho(x)), \quad x\in C^*(\A),
\]
and consequently $\hat{\rho}(C^*_{\red}(\A)) \subset C^*_{\red}(\A)$. 

We now show that $\hat{\rho}$ is injective. Since $\A$ is completely rational, the localized representation
$\pi_0 \circ \rho$ has a direct summand $\pi_0 \circ \sigma$ with $\sigma$ a localized covariant endomorphism of $C^*(\A)$ with finite statistical dimension. Let $\bar{\sigma}$ be the corresponding conjugate endomorphism.  
Then $\pi_{\red}\circ \bar{\sigma}\rho$ contains a subrepresentation unitarily equivalent to $\pi_{\red}$. On the other hand, 
the locally normal representation $\pi_{\red}\circ \bar{\sigma}$ is quasi-equivalent to a subrepresentation of $\pi_{\red}$ 
so that $\pi_{\red}\circ \bar{\sigma}\rho$ is quasi-equivalent to a subrepresentation of  $\pi_{\red}\circ \rho$. Accordingly 
$\pi_{\red}$ is quasi-equivalent to a subrepresentation of $\pi_{\red}\circ \rho$. As a consequence the identical representation of $C^*_{\red}(\A)$ on $\H_{\red}$ is quasi-equivalent to a subrepresentation of the representation
$\hat{\rho}$ and hence the latter is faithful.

Let $e\in C^*_{\red}(\A)$ be any orthogonal projection with finite-dimensional range $e\H_{\red}$. Then 
$e$ is a finite projection in the von Neumann algebra $C^*_{\red}(\A)$ and hence $\hat{\rho}(e)$ is finite in 
$\hat{\rho}(C^*_{\red}(\A))$. If $\rho$ has finite  statistical dimension then $\hat{\rho}(\pi_{\red}(C^*(\A)))'=
\pi_{\red}(\rho(C^*(\A)))'$ is finite-dimensional and hence $\hat{\rho}(e)$ has finite-dimensional range. It follows that 
\[
\hat{\rho}(\KK) \subset \pi_{\red}(C^*(\A)) \cap \K(\H_{\red}) = \KK.
\]
\end{proof}

The nondegenerate irreducible representations of $\KK$ are in one-to-one correspondence with the sectors of $\A$; moreover, while $C^*_{\locn}(\A)$ cannot be a good candidate for a K-theoretical description of the superselection structure of chiral CFTs (see Theorem \ref{K=0}), $\KK$ will turn out in Section \ref{sec:K-th} to be fine.
In contrast to its nice C*-algebraic properties, however, $\KK$ completely forgets the net structure of $\A$, in the sense that it has zero intersection with any local algebra $\A(I)$, $I\in \I$.

\section{Superselection sectors and K-theory}\label{sec:K-th}

\subsection*{K-theory basics}

Let $B$ be a $*$-algebra. We say that $p \in B$ is a projection if it is a self-adjoint idempotent, i.e. $p=p^*$ and $p^2=p$. 
We say that two projections $p, q \in B$ are (Murray - von Neumann) equivalent if there exists a $w \in A$   
such that $w^*w =p$ and $ww^* =q$. If this is the case we write $p \sim q$. We denote by $\operatorname{Proj}(B)$ the set of projections of $B$ and by $V_0(B)$ the set $\operatorname{Proj}(B)/\sim$ of equivalence classes of projections of $B$. 

Let $A$ be a \emph{stably unital} C*-algebra, \ie a C*-algebra with an approximate unit of projections; denote its unitization by $A^+$ if it is not unital.  While details (and the appropriate definitions when $A$ is not stably unital) can be found in \cite{Bl} or \cite[Sect.V.1]{Bla06}, we provide here only the main
 
\begin{definition}\label{def:prelim-K}
\begin{itemize}
\item[$(K_0)$] Let $\Mat_r(A)$ be the  C*-algebra of $r\times r$ matrices over $A$, $r\in \N$, and let $\Mat_\infty(A)$ denote the 
$*$-algebra of infinite matrices over $A$ with only finitely many nonzero entries. The 
maps $x\mapsto {diag}(x,0)$ define natural embeddings $\Mat_r(A) \ra \Mat_{r+1}(A)\ra ...\ra \Mat_\infty(A)$, giving rise to a directed system of C*-algebras. Let $V(A)  := V_0(\Mat_\infty(A))$. There is a binary operation 
$$(p_1,p_2)\in \operatorname{Proj}\left(\Mat_{r_1}(A)\right)\times \operatorname{Proj}\left(\Mat_{r_2}(A)\right) \mapsto  diag(p_1, p_2) \in 
\operatorname{Proj}\left(\Mat_{r_1+r_2}(A)\right) ,$$ which turns $V(A)$ into an abelian semigroup. Then the 
\emph{$K_0$-group of $A$} is defined as
\[
  K_0(A) := \textrm{Grothendieck group of } V(A),
\]
where the Grothendieck group of an arbitrary additive semigroup $H$ is the group of formal differences of elements of 
$H$, i.e.
\[
  (H\times H) /\{(h_1,h_2) \sim_{H\times H} (g_1,g_2)\; \Leftrightarrow \; (\exists k\in H) h_1+g_2 +k = h_2 + g_1 +k \}.
\]
We write $[p]$ for the element in $K_0(A)$ induced by a projection $p\in A$.
\item[$(K_1)$] Let $U_r(A)$ denote the set of unitary elements in $\Mat_r(A^+)$ of the form $\unit_r+x$ with $x\in\Mat_r(A)$, and let $U_r(A)_0$ be the connected component of $\unit_r$ in $U_r(A)$.  With the diagonal inclusion $u\in  U_r(A) \mapsto u\oplus 1 \in  U_{r+1}(A)$, this gives a directed system, and with the operation $(u_1,u_2)\in U_{r_1}(A)\times U_{r_2}(A) \mapsto u_1\oplus u_2\in U_{r_1+r_2}(A)$, the union of all $U_r(A)$ becomes a group. Then the \emph{$K_1$-group of $A$} is defined as the direct limit
\[
  K_1(A) := (\lim_{\ra} U_r(A))/ (\lim_{\ra}U_r(A)_0).
\]
\end{itemize}
\end{definition}
 
\begin{remark}\label{MrA} Let $w \in \Mat_\infty (A)$ be such that $w^*w=p$ and $ww^*=q$, where $p, q \in \Mat_r(A)$ are 
projections. Then $w=qwp \in \Mat_r(A)$. Hence, if two projections in $\Mat_r(A)$ are equivalent in 
$\Mat_\infty(A)$, then they are equivalent in $\Mat_r(A)$. Accordingly, for every positive integer 
$r$, there is an embedding of the set $V_0(\Mat_r(A))$ of equivalence classes of projections of $\Mat_r(A)$ into 
$V(A)=V_0(\Mat_\infty(A))$ giving rise to natural maps $V_0(\Mat_r(A)) \to K_0(A)$. In particular, the embedding of 
$A$ in $\Mat_\infty(A)$ gives rise to a natural map from $V_0(A)$ into $K_0(A)$. Finally, one can use 
$A\otimes \K$ instead of $\Mat_\infty(A)$. In fact $V(A)=V_0(A \otimes K)$, cf. \cite[Sect. 5.1]{Bl}.
\end{remark}

K-theory is a functor from the category of stably unital C*-algebras to abelian groups. Moreover, it can be regarded as a cohomology theory on C*-algebras; in particular, it satisfies the additivity property, namely
\[
K_*(A\oplus B) = K_*(A) \oplus K_*(B),
\]
for C*-algebras $A,B$. It defines an important invariant for the description and classification of C*-algebras.

Every $*$-homomorphism $\phi : A \to B$ induces a well-defined group homomorphism \linebreak 
$\phi_*: K_*(A) \to K_*(B)$, called the push-forward, with the property $\phi_*([p])=[\phi(p)]$, for every 
$p \in \operatorname{Proj}(A)$. It is compatible with composition, namely $(\phi \psi)_*=\phi_*\psi_*$. 

One might be tempted to apply this property of $*$-homomorphisms in order to achieve a K-theoretic interpretation of the representation theory of conformal nets on $S^1$ and its localized endomorphisms. However, since infinite factors have trivial K-theory 
\cite[5.3.2. (b), 8.1.2 (b)]{Bl},  as a straightforward consequence of additivity of K-theory and Theorem \ref{piC*A_rational} we have

\begin{theorem}\label{K=0}
Let $\A$ be a completely rational conformal net on $S^1$. Then \linebreak $K_*(\pi(C^*(\A))) = 0$, for every locally normal representation $\pi$ of $\A$. In particular, \linebreak $K_*(C^*_{\locn}(\A)) = 0$.
\end{theorem}

We shall see, however, that one can overcome these problems by considering the  C*-subalgebra 
$\KK\subset C^*_{\red}(\A)$ introduced in Proposition \ref{th:univ-K-inCA}. 
The purpose of the following subsection is to show that we get indeed a nontrivial semiring action of the fusion ring $\Ring$ (\cf Remark \ref{rem:CFT-univC} (4) for definition and some properties) on $K_*(\KK)$.

\subsection*{Fusion semiring action on $K_0(\KK)$}

In contrast to the algebra $C^*_{\locn}(\A)$, which has trivial K-theory, we have
\begin{equation}\label{eq:KK-K}
K_0(\KK)\simeq \Z^n, \quad K_1(\KK)=0.
\end{equation}

This follows from additivity because $\KK\simeq \K^{\oplus n}$ and 
\[
K_0(\K)=\Z,\quad K_1(\K)=0,
\]
actually we have $V_0(\K)=V(\K) = \N_0$ \cf \cite[V.1.1.10\&1.16\&2.3\&2.7]{Bla06}. 

Note that, in the identification $\N_0 = V(\K)$,  
$1 \in \N_0$ must correspond to the class of an arbitrary minimal projection in $\Mat_\infty(\K)$ and hence to the class of 
an arbitrary minimal projection in $\K$. Accordingly $K_0(\K)$ is generated by $V_0(\K)$ through the embedding 
(in fact surjective) $V_0(\K) \subset V(\K)$, cf. Remark \ref{MrA}. Similarly, $V_0(\KK)$ generates $K_0(\KK)$. Actually 
the $n$ equivalence classes in $V_0(\KK)$ corresponding to the minimal (one-dimensional) projections in $\KK$ generate $K_0(\KK)$.

By the above facts nontrivial actions on $K_0(\KK)$ are not {\it a priori} excluded. In fact, we have

\begin{theorem}\label{th:K-sectors}
Let $\A$ be a completely rational net on $S^1$. Then 
the push-forward of \linebreak $*$-homomorphisms of $\KK$ gives rise to a faithful semiring action of the fusion semiring $\Ring$ on $K_0(\KK)$, 
namely an injective semiring homomorphism $\eta:\Ring\ra \End(K_0(\KK))$ satisfying $\eta_{[\rho]} = (\hat{\rho}|_{\KK})_*$ for every 
localized covariant endomorphism $\rho$ of $C^*(\A)$. 
It corresponds to the regular representation of the fusion algebra associated to $\Ring$. 
\end{theorem}

\begin{proof}
Every localized endomorphism $\rho$ of $C^*(\A)$ with finite statistical dimension localized in an arbitrary but fixed interval $I_0\in\I$ induces a normal $*$-homomorphism $\hat{\rho}$ of $C^*_{\red}(\A)$ satisfying $\hat{\rho}\circ\pi_{\red}=\pi_{\red}\circ\rho$ (\cf Proposition \ref{th:univ-K-inCA} (2)); moreover, it preserves the subalgebra $\KK\subset C^*_{\red}(\A)$. Let $\rho_i$, for $i=1,...,n$, be endomorphisms of 
$C^*(\A)$ all covariant and localized in $I_0$, with $\rho_1$ is the identity endomorphism $\id$ and such that $\pi_0 \circ \rho_i$, $i=1,...,n$, are pairwise inequivalent irreducible representations. Accordingly we can assume that 
$$\pi_{\red}=\bigoplus_{i=1}^{n} \pi_0 \circ \rho_i .$$

We have to show that there is an injective semiring homomorphism
\[
\eta:\Ring\ra \End(K_0(\KK)),
\]
which will then define the action. We proceed in four steps.

(1) \emph{Construction and well-definedness of $\eta$.} Notice that every element in $\KK$ can be written as $\pi_{\red}(x)$, with a suitable $x\in C^*(\A)$. 
Let $\rho$ and $\sigma$ be unitarily equivalent covariant endomorphisms of $C^*(\A)$ localized in $I_0$. Accordingly, there is 
a unitary $u\in C^*(\A)$ such that $\sigma = \Ad(u)\circ \rho$.  Hence we find
\begin{align*}
 \hat{\sigma}(\pi_{\red}(x))=& \pi_{\red}(\sigma(x)) = \pi_{\red}(u\rho(x)u^*) \\
=&  \pi_{\red}(u) \pi_{\red}(\rho(x)) \pi_{\red}(u)^* = \pi_{\red}(u) \hat{\rho}(\pi_{\red}(x)) \pi_{\red}(u)^*, \quad x\in C^*(\A). 
\end{align*}
This shows that $\hat{\sigma}$ and $\hat{\rho}$ are unitarily equivalent in $C^*_{\red}(\A)$ via $u_{\red}:=\pi_{\red}(u)$. The $*$-homomorphisms $\hat{\rho}$ and $\hat{\sigma}$ act in K-theory by push-forward. For a projection $p\in \KK$, we have
$\hat{\sigma}(p)=u_{\red}\hat{\rho}(p)u_{\red}^*$. Set $w=\hat{\rho}(p)u_{\red}^* \in \KK$. 
Then $ww^*= \hat{\rho}(p)$ and $w^*w=  \hat{\sigma}(p)$ and hence  $(\hat{\sigma}|_{\KK})_*([p])=(\hat{\rho}|_{\KK})_*([p])$.      
Since $K_0(\KK)$ is generated by equivalence classes of projections in $\KK$, we see that 
$(\hat{\rho}|_{\KK})_*=(\hat{\sigma}|_{\KK})_*$ on all $K_0(\KK)$; in other words, $(\hat{\rho}|_{\KK})_*$ depends only on the equivalence class $[\rho]\in\Ring$ of $\rho$, and

\begin{equation}\label{eq:def-alpha}
\eta: [\rho]\in\Ring \mapsto (\hat{\rho}|_{\KK})_* \in \End(K_0(\KK))
\end{equation}
is well-defined.

(2) \emph{Additivity.} Since $\eta_{[\rho]}$ is a group homomorphism, it suffices to check additivity of the action on the generators 
in $V_0(\KK)$ of $K_0(\KK)$. Given $[\rho],[\sigma]\in\Ring$, recall from Remark \ref{rem:CFT-univC} $(4)$ that their sum in $\Ring$ is defined as
\[
[\rho]+[\sigma] = [\lambda],
\]
where $\lambda$ is any covariant endomorphism of $C^*(\A)$ localized in $I_0$ and such that $\pi_0 \circ \lambda$ is equivalent to  
$\pi_0\circ \rho \oplus \pi_0 \circ \sigma$. Then we can write 
$\lambda=v_1\rho(\cdot) v_1^* + v_2 \sigma(\cdot)v_2^*$
with isometries $v_1,v_2\in\A(I_0)$ defining a Cuntz algebra as a unital subalgebra of $\A(I_0)$. 
Since, for any $i$, $v_{i,\red}:= \pi_{\red}(v_i)$ is again an isometry (now in $C^*_{\red}(\A)$), given a minimal projection 
$p\in\KK$, $v_{i,\red}p v_{i,\red}^*$ is again a minimal projection in $\KK$ having the same central support of $p$ in 
$C^*_{\red}(\A)$.  Accordingly $p$ and $v_{i,\red}p v_{i,\red}^*$ are equivalent in $C^*_{\red}(\A)$ and hence they are equivalent in 
the ideal $\KK$. 
Therefore
\begin{align*}
 \eta_{[\rho]+[\sigma]}([p])
=& \eta_{[v_1\rho(\cdot) v_1^* + v_2 \sigma(\cdot)v_2^*]}([p])\\
=& [v_{1,\red}\hat{\rho}(p) v_{1,\red}^* +  v_{2,\red} \hat{\sigma}(p) v_{2,\red}^*]\\
=& [v_{1,\red}\hat{\rho}(p) v_{1,\red}^*] +  [v_{2,\red} \hat{\sigma}(p) v_{2,\red}^*]\\
=& [\hat{\rho}(p)] + [\hat{\sigma}(p)] \\
=& \eta_{[\rho]}([p]) +\eta_{[\sigma]}([p]),
\end{align*}
\ie additivity holds.

(3) \emph{ Multiplicativity.} According to the above general discussion of push-forwards of 
$*$-homomorphisms to K-theory, we always have, for $[\rho],[\sigma]\in\Ring$:
\[
\eta_{[\rho]} \circ \eta_{[\sigma]} = (\hat{\rho}|_{\KK})_* \circ (\hat{\sigma}|_{\KK})_*
= ((\hat{\rho}\hat{\sigma})|_{\KK})_*= \eta_{[\rho] [\sigma]}.
\]

Summing up, $\eta:\Ring\ra \End(K_0(\KK))$ is a well-defined nontrivial homomorphism of semirings.

(4) \emph{Faithfulness and regular representation.} Recall the direct sum decomposition
\[
\K_1\oplus ...\oplus \K_n = \KK \subset C^*_{\red}(\A) =\bigoplus_{l=1}^n \pi_0\circ\rho_l (C^*(\A)).
\]

Let $z_1, z_2,\dots , z_n$ be the minimal projections of ${\mathcal Z}(\KK)$ ordered in the natural way. 
Accordingly $ z_i$ is the central support corresponding to the subrepresentation $\pi_0 \circ \rho_i$ for 
$k=1,\dots , n$. Now let $\tilde{z}_i \in {\mathcal Z}(W^*(\A))$ be the central support corresponding to any subrepresentation 
of the universal representation of $C^*(\A)$ unitarily equivalent to $\pi_0 \circ \rho_i$. Then, if $\tilde{\pi}_{\red}$ denotes the unique 
normal extension of $\pi_{\red}$ to the enveloping von Neumann algebra $W^*(\A)$ of $C^*(\A)$ then 
$\tilde{\pi}_{\red}(\tilde{z_i})=z_i$. More generally, for every representation
 $\pi$ of $C^*(\A)$ on a Hilbert space $\H_\pi$, $\tilde{\pi}(\tilde{z}_i)$ is the central support corresponding to any 
 subrepresentation of $\pi$ equivalent to $\pi_0 \circ \rho_i$. Now recalling that for $k=1, \dots, n$ we have 
 $\hat{\rho}_k \circ \pi_{\red} = \pi_{\red} \circ \rho_k$ we can conclude that $\hat{\rho}_k(z_1)$ is the central support 
 corresponding to any subrepresentation of $\pi_{\red} \circ \rho_k$ unitarily equivalent to the vacuum representation $\pi_0$.
 But, by the properties of the conjugate endomorphisms  $\pi_0$ is equivalent to a subrepresentation of 
 $\pi_0\circ \rho_i \bar{\rho}_k$ iff $i =k$. Moreover, for $i=k$, the multiplicity is one. Accordingly, $\hat{\rho}_k(z_1)$ is a minimal 
 projection in $\hat{\rho}_k(C^*_{\red}(\A))'$ and $\hat{\rho}_k(z_1) \leq z_k $.  It follows that if $p_1$ is a minimal projection in 
 $\KK$ with nonzero entry in $\K_1$, i.e. $p_1 \leq z_1$, then  $\hat{\rho}_k(p_1) = \hat{\rho}_k(p_1)\hat{\rho}_k(z_1)$ is 
 a minimal projection in $\KK$ such that $\hat{\rho}_k(p_1)z_k = \hat{\rho}_k(p_1)$.

The equivalence class in $K_0(\KK)$ of $\hat{\bar{\rho}}_k(p_1)$ does not depend on the actual choice of the minimal projection and is consequently precisely the $k$-th generator of $K_0(\KK)\simeq \Z^n$. In other words, $[\hat{\bar{\rho}}_k(p_1)]$, with $k=1,...,n$, is the canonical basis of $K_0(\KK)$. Notice that involutivity of conjugation on the equivalence classes implies that, for every $l$ there is a unique $k$ such that 
$[\bar{\rho}_l]=[\rho_k]$. This shows that
\[
e_k:=[\hat{\rho}_k(p_1)] = \eta_{[\rho_k]}(e_1), \quad k=1,...,n,
\]
defines again a basis of $K_0(\KK)$, which is just a permutation of the canonical one.

With this choice, $\eta$ is easily seen to correspond to the regular representation of the fusion rules of $\Ring$, namely for $i,j=1,...,n$, we have
\begin{equation}\label{eq:KK-leftreg}
\begin{aligned}
\eta_{[\rho_i]}(e_j) =&\eta_{[\rho_i]}\eta_{[\rho_j]}(e_1)
= \eta_{[\rho_i]\cdot[\rho_j]} (e_1)
= \eta_{\sum_{k=1}^n N_{i,j}^{k} [\rho_k]}(e_1)\\
=& \sum_{k=1}^n N_{i,j}^{k} \eta_{[\rho_k]}(e_1)
= \sum_{k=1}^n N_{i,j}^{k} e_k .
\end{aligned}
\end{equation}

Here we made use of additivity and multiplicativity of $\eta$, and of the fusion matrix introduced in \eqref{eq:prelim-fusion1}. In particular, for $j=1$ and arbitrary $[\rho]\in\Ring$, which can always be uniquely decomposed as $\sum_{i=1}^n m_i [\rho_i]$ with certain 
$m_i\in\NN$, we find 
\[
\eta_{[\rho]}(e_1)= \sum_i m_i e_i,
\]
so the map $[\rho]\in\Ring \mapsto \eta_{[\rho]}\in\End(K_0(\KK))$ is actually injective.

Now let $\RRing$ be the Grothendieck group of the additive semigroup $\Ring$. From the above discussion we see that the 
the map $[\rho] \mapsto \eta_{[\rho]}(e_1)$ is a semigroup isomorphism $\beta$ of additive semigroups from $\Ring$ into $K_0(\KK)$ whose extension to $\RRing$ gives rise to a surjective group isomorphism $\tilde{\beta} : \RRing \to K_0 (\KK)$ and the action $\eta$ of 
$\Ring$ uniquely extends to an action $\tilde{\eta} : \RRing \to \End (K_0(\KK))$. 
Now, $\C \otimes_{\Z} \RRing$ is a fusion algebra with distinguished basis 
$1\otimes_\Z [\rho_1] \dots 1\otimes_\Z [\rho_n]$, see e.g. \cite{Yam99} for the definition, and $\id \otimes \tilde{\eta}$ 
is a representation of $\C \otimes_{\Z} \RRing$ on the $n$-dimensional complex vector space $\C \otimes_{\Z} K_0(\KK)$. 
For $[\sigma] , [\rho] \in \Ring \subset \RRing$ we have 

$$\left(\id \otimes_{\Z} \tilde{\beta}\right)^{-1}\left(\id \otimes_{\Z} \tilde{\eta}_{[\sigma]} \right)\left(\id \otimes_{\Z} \tilde{\beta} \right)
\left( 1 \otimes_{\Z} [\rho] \right) = [\sigma] [\rho] $$
and hence 
$$\left(\id \otimes_{\Z} \tilde{\beta}\right)^{-1}\left(\id \otimes_{\Z} \tilde{\eta}\right)_x \left(\id \otimes_{\Z} \tilde{\beta} \right) y
 = x y $$
for all $x, y \in \C \otimes_{\Z} \RRing$, namely $\left(\id \otimes_{\Z} \tilde{\beta}\right)^{-1}\left(\id \otimes_{\Z} \tilde{\eta}\right) \left(\id \otimes_{\Z} \tilde{\beta} \right)$ is the regular representation of the commutative fusion algebra $\C \otimes_{\Z} \RRing$.
  \end{proof}

The group isomorphism $\RRing \simeq K_0(\KK)$ might be used in order to define a fusion ring structure on $K_0(\KK)$. Obviously this structure is not encoded in the  C*-algebra $\KK$ on its own but depends on the underlying net $\A$ through the action on 
$\KK$ of the DHR endomorphisms.  If $\A_1$ and $\A_2$ are two completely rational conformal nets with the same number of sectors 
then ${\mathfrak K}_{\A_1}$ and ${\mathfrak K}_{\A_2}$ are isomorphic  C*-algebras but typically 
$\tilde{\mathcal R}_{\A_1}$ and $\tilde{\mathcal R}_{\A_2}$ are not isomorphic fusion rings. 

\vspace{
\baselineskip}
 
{\small
 
\noindent\textbf{Acknowledgements.} 

We thank R. Longo for bringing to our attention some of the issues discussed in the main text and for useful discussions. We also thank M. Yamashita for comments on an earlier draft of this paper.  Three of us (S. C., R. H. and M. W.) benefitted of a visit at RIMS, Kyoto University in occasion of the program ``Operator algebras and their applications". They thank the organizers for the invitation.

}
 
\end{document}